\documentclass[12pt,a4paper]{amsart}
\usepackage{amssymb,latexsym,enumerate,url,hyperref} 

\newcommand{\CC}{{\mathbb{C}}}
\newcommand{\FF}{{\mathbb{F}}}
\newcommand{\ZZ}{{\mathbb{Z}}}

\newcommand{\bC}{{\mathbf C}}
\newcommand{\bG} {\mathbf G}
\newcommand{\bL} {\mathbf L}

\newcommand{\bT} {\mathbf T}
\newcommand{\bZ}{{\mathbf Z}}
\newcommand{\bO}{{\mathbf O}}

\newcommand{\cE} {\mathcal E}

\newcommand{\fA} {\mathfrak A}
\newcommand{\fS} {\mathfrak S}

\newcommand{\Aut}{{{\operatorname{Aut}}}}
\newcommand{\df}{{{\operatorname{def}}}}
\newcommand{\diag}{{{\operatorname{diag}}}}
\newcommand{\Id}{{{\operatorname{Id}}}}
\newcommand{\Irr}{{{\operatorname{Irr}}}}
\newcommand{\Tr}{{{\operatorname{Tr}}}}
\newcommand{\GL}{\operatorname{GL}}
\newcommand{\PGL}{\operatorname{PGL}}
\newcommand{\GaL}{\operatorname{\Gamma L}}
\newcommand{\SL}{\operatorname{SL}}
\newcommand{\PSL}{\operatorname{PSL}}
\newcommand{\GU}{\operatorname{GU}}
\newcommand{\SU}{\operatorname{SU}}
\newcommand{\Sp}{\operatorname{Sp}}

\newcommand{\Spin}{\operatorname{Spin}}
\newcommand{\OO}{\operatorname{O}}
\newcommand{\SO}{\operatorname{SO}}

\newcommand{\Stab}{\operatorname{Stab}}
\newcommand{\Trace}{\operatorname{Trace}}

\newcommand{\tw}[1]{{}^{#1}\!}
\renewcommand{\sp}[1]{{<\!#1\!>}}

\let\eps=\epsilon
\let\la=\lambda
\let\si=\sigma
\let\tht=\theta
\let\vhi=\varphi

\newtheorem{thm}{Theorem}[section]
\newtheorem{lem}[thm]{Lemma}
\newtheorem{conj}[thm]{Conjecture}
\newtheorem{prop}[thm]{Proposition}

\newtheorem{cor}[thm]{Corollary}

\newtheorem*{conjA}{Conjecture A}
\newtheorem*{thmB}{Theorem B}

\newtheorem*{thmC}{Theorem C}
\newtheorem*{thmD}{Theorem D}

\theoremstyle{definition}
\newtheorem{rem}[thm]{Remark}

\newtheorem{exmp}[thm]{Example}

\numberwithin{equation}{section}

\def\irr#1{{\rm Irr}(#1)}
\def\ibr#1{{\rm IBr}(#1)}
\def\cent#1#2{{\bf C}_{#1}(#2)}
\def\nor{\trianglelefteq\,}

\begin{document}

\title[Zeros of characters and element orders]{Zeros of characters and orders\\ of elements in finite groups}

\author{Gunter Malle}
\address{FB Mathematik, RPTU Kaiserslautern, Postfach 3049,
  67653 Kaisers\-lautern, Germany.}
\email{malle@mathematik.uni-kl.de}
\author{Gabriel Navarro}
\address{Departament de Matem\`atiques, Universitat de Val\`encia, 46100
  Burjassot, Val\`encia, Spain}
\email{gabriel@uv.es}
\author{Pham Huu Tiep}
\address{Department of Mathematics, Rutgers University, Piscataway, NJ 08854,
  USA}
\email{tiep@math.rutgers.edu}

\thanks{The first author gratefully acknowledges financial support by the DFG,
  Project-ID 286237555 -- TRR 195. The work of the second author is supported
  by Grant PID2022-137612NB-I00 funded by
  MCIN/AEI/10.13039/501100011033 and ERDF ``A way of making Europe.” The third
  author gratefully acknowledges the support of the NSF (grant DMS-2200850),
  the Joshua Barlaz Chair in Mathematics, and the Charles Simonyi Endowment at
  the Institute for Advanced Study (Princeton).}



\date{\today}

\begin{abstract}
We investigate a beautiful conjecture of T.~Wilde on character values and
element orders of finite groups. We reduce it to a statement on nearly simple
groups that can be checked ``prime by prime". For these groups, we show that a
strong form of Wilde's conjecture holds in many important cases, and for primes
$p>5$ we are able to show the required statement for most classes of nearly
simple groups. The few remaining cases, however, seem to require information on
extensions of irreducible characters that are not available at the present
time.
\end{abstract}

\maketitle

\pagestyle{myheadings}

\section{Introduction}
 
The first theorem on zeros of characters of finite groups is due to W.~Burnside
and asserts that every non-linear irreducible character $\chi$ of a finite
group $G$ takes the value zero on some element (see \cite[Thm~3.15]{Is}).
(Much later on, it was proved in \cite{MNO}, using the classification of finite
simple groups, that this zero can be chosen to occur on an element of
prime power order.) R.~Brauer and C.~Nesbitt proved the most celebrated
connection between the order of elements and the \emph{codegree} $|G|/\chi(1)$
of $\chi$: if a prime $p$ does not divide the integer $|G|/\chi(1)$, then
$\chi(g)=0$ for all $g\in G$ of order divisible by~$p$ (see
\cite[Thm~8.17]{Is}). A more sophisticated version of this involving defect
groups and blocks is a classical consequence of Brauer's second main theorem,
see below.
\medskip

This paper investigates an intriguing conjecture of T.~Wilde \cite{Wi06}, which
relates orders of elements, zeros, and codegrees, implies the result of Brauer
and Nesbitt, and whose statement looks innocent.

\begin{conjA}[Wilde]
 Let $G$ be a finite group and let $\chi\in\Irr(G)$. If $\chi(g)\ne 0$ for
 some $g\in G$, then the order $o(g)$ of $g$ divides $|G|/\chi(1)$.
\end{conjA}

Our first main result is the reduction of Conjecture A to a stronger version of
the problem for nearly simple groups. Let us say that $(H,h,\chi)$
\emph{satisfies $(^*)$} if $H$ is a finite group, $h\in H$, the derived
subgroup $L=[H,H]$ is quasi-simple, $\bZ(L)=\bZ(H)$ is cyclic,
$H=L\langle h\rangle$, and $\chi \in \Irr(H)$ is faithful with $\chi(h)\ne 0$.
 
\begin{thmB}
 Assume that 
 \begin{equation}   \label{condB}
   o(h\bZ(H)) \quad \text{divides} \quad \frac{|H:\bZ(H)|}{\chi(1)}
 \end{equation}
 for all $(H,h,\chi)$ satisfying $(^*)$. Then Conjecture~A is true.
\end{thmB}

(Recall that $|H:\bZ(H)|/\chi(1)$ is always an integer by \cite[Thm~3.12]{Is}.)
So the question now is whether we can prove that nearly simple groups that
satisfy $(^*)$ also satisfy Condition~\eqref{condB}.
\medskip

Of course, Conjecture A, and Condition~\eqref{condB} in Theorem~B, can be
formulated (and checked) one prime at a time: if $\chi(g)\ne 0$, then we only
need to prove that $o(g)_p$ divides $(|G|/\chi(1))_p$, where $n_p$ is the
largest power of~$p$ dividing the integer~$n$. Indeed, in what follows we will
frequently work with Condition~\eqref{condB} at a fixed prime~$p$.
(However, a word of caution is necessary here: we cannot
replace $g$ by the $p$-part~$g_p$ of $g$, since $\chi(g)\ne 0$ does not imply
that $\chi(g_p) \ne 0$. An example of this is $\fA_{10}$, $p=2$, $o(g)=6$ and
$\chi(1)=84$.) This allows us to bring modular representation theory to the
problem. 
\medskip

An easy consequence of Brauer's second main theorem states that if
$\chi(g)\ne0$ for some $g\in G$ and $\chi\in\Irr(G)$ in a Brauer $p$-block $B$
of~$G$, then $g_p$ lies in some defect group of~$B$ (see \cite[Cor.~5.9]{Na98}).
Thus, Conjecture~A holds for~$G$ (\emph{at the prime $p$}) if whenever
$\chi\in\Irr(G)$ lies in a $p$-block $B$ of $G$ with defect group~$D$, then the
exponent $\exp(D)$ of~$D$ satisfies
$$\exp(D) \le \left(\frac{|G|}{\chi(1)}\right)_p \, . \eqno{(\ddagger)}$$
If this is satisfied for all $\chi\in\Irr(B)$ we say \emph{$B$ satisfies
Condition $(\ddagger)$}, and if it holds for all $p$-blocks of $G$, we say
\emph{$G$ satisfies Condition $(\ddagger)$ at $p$}. Recalling that the
\emph{defect $\df(\chi)$} of $\chi$ is defined
by $p^{\df(\chi)}=(|G|/\chi(1))_p$, this may be reformulated as saying that
$$\exp(D)\le p^{\df(\chi)}\qquad\text{for all $\chi\in\Irr(B)$}.$$
It is not the first time that orders of elements in a defect group and defects
of characters are related: it is a deep theorem of Brauer \cite[(6C)]{B} that
$$\exp(\bZ(D))\le p^{\df(\chi)}\qquad\text{for all $\chi\in\Irr(B)$}.$$
In fact, by a conjecture of Robinson, whose proof was recently completed in
\cite{KM25} using the classification of finite simple groups, it is even true
that
$$|\bZ(D)|\le p^{\df(\chi)}\qquad\text{for all $\chi\in\Irr(B)$}.$$
Of course, neither of these imply $(\ddagger)$ in general.

Although Condition~$(\ddagger)$ is frequently satisfied, for example whenever
$D$ is abelian by Brauer's result, unfortunately it does not always hold.
The easiest example is $\GL_2(3)$ for $p=2$, but there are instances of this
for all primes as we will see in Example~\ref{gunterexmp}.
(We remark here that for $p$-solvable groups and primes $p>3$,
Condition~$(\ddagger)$ always holds by \cite[Thm~A]{MN22}.)

\smallskip
In fact, we will also consider the following variant of Condition $(\ddagger)$,
which is equivalent to $(\ddagger)$ if the Sylow $p$-subgroup $\bZ(G)_p$ of
$\bZ(G)$ is trivial:
$$\exp(D/\bZ(G)_p) \le \left(\frac{|G:\bZ(G)|}{\chi(1)}\right)_p \, .
  \eqno{(\ddagger^\star)}$$
The importance of Condition $(\ddagger^\star)$ for us is that if a quasi-simple
group $L$ satisfies it at $p$, then each $(H,h,\chi)$ with $[H,H]=L$
satisfying~$(^*)$ fulfils \eqref{condB} at $p$. (See
Proposition~\ref{prop:to qs} below.)
This allows us to prove Wilde's conjecture in many cases.

\begin{thmC}
 Suppose that $L$ is a quasi-simple group and $p$ is a prime. Then $L$
 satisfies Conditions~$(\ddagger)$ and $(\ddagger^\star)$ both at $p$ in the
 following cases:
 \begin{enumerate}[\rm(1)]
  \item $L=\fA_n$ for $n\ge5$ and all $p$;
  \item $L=2.\fA_n$ for $n\ge5$ and $p>2$;
  \item $L/\bZ(L)$ is simple of Lie type in characteristic not $p$, $p>2$, and
   $p{\not|}(n,q-\eps)$ when $L/\bZ(L)=\PSL_n(\eps q)$, and $p>5$ if $L/\bZ(L)$
   is of exceptional type;
  \item $L/\bZ(L)$ is simple of Lie type in characteristic~$p$, where $p>2$ if
   $L$ is an exceptional covering group; or
  \item $L/\bZ(L)$ is a sporadic simple group and $p>2$;
 \end{enumerate}
 \end{thmC} 

When Condition $(\ddagger^\star)$ fails for a certain quasi-simple group, then
we are forced to work directly with nearly simple groups $H$ satisfying
Condition~$(^*)$ to prove Condition~\eqref{condB} in Theorem~B. This is done
in a number of cases; in particular see Proposition~\ref{prop:sl2}
in the critical case $2<p|(n,q-1)$ for $[H,H] = \SL_n(q)$. 

\begin{thmD}
 Condition~\eqref{condB} holds for all groups $H$ with associated simple
 group $S$ one of $\fA_n$ ($n\ge5$), $\tw2B_2(q^2)$, $^2G_2(q^2)$, or a
 sporadic simple group. Moreover, it holds for $S$ simple of Lie type in
 characteristic~$p$ at the prime~$p$.
\end{thmD}

As a consequence of our results in Sections~\ref{sec:oneprime}
and~\ref{sec:Lie}, when $p>5$ the only exceptions to Condition~\eqref{condB}
at the prime~$p$ on nearly simple groups satisfying $(*)$ are very special
types of $p$-blocks of decorated versions of $\PSL_n(\eps q)$ (described in
Corollary~\ref{cor:SLn}). And, as illustrated by the proof of
Proposition~\ref{prop:sl2}, Condition~\eqref{condB} would probably follow even
in these exceptional cases given a good control over the Lusztig restriction
of characters of disconnected groups.
\medskip

The paper is structured as follows. In Section~\ref{sec:red} we prove
Theorem~\ref{thm:reduction} giving a reduction of a strong form of Wilde's
conjecture to a statement on nearly simple groups (see
Conjecture~\ref{conj:aqs}, as well as its equivalent reformulation in
Conjecture~\ref{conj:aqs12}), and Theorem B. In Section~\ref{sec:oneprime} we
verify Conjecture~\ref{conj:aqs12} at a fixed prime via 
Condition~$(\ddagger)$ for nearly simple groups, in the cases where the
associated simple group $S$ is an alternating group, a sporadic simple group, 
or a simple group of Lie type with exceptional Schur multiplier. The case of
the remaining simple groups of Lie type for~$S$ is studied in
Section~\ref{sec:Lie}, where we also prove Theorems C and D.

\section{Reductions to Nearly Simple Groups}   \label{sec:red}

In this section, we prove that if nearly simple groups satisfy
Conjecture~\ref{conj:aqs} below, then a stronger version of Conjecture~A is
true. Recall that for $G$ a group, $G^{(\infty)}$ is the last term in the
derived series; that is, the smallest normal subgroup with solvable factor
group. We use the notation in \cite{Is} and \cite{Na18}. For instance, if
$N\nor G$ and $\theta \in \Irr(N)$, then $\irr{G|\theta}$ is the set of
irreducible constituents of the induced character $\theta^G$, which is also the
set of irreducible characters $\chi$ of $G$ whose restriction to~$N$ contains
$\theta$. Restriction of a character $\chi$ of $G$ to a subgroup $H$ is denoted
by either $\chi_H$ or $\chi|_H$. Notice that if $\chi \in \irr{G|\theta}$, then
$\chi(1)/\theta(1)$ divides $|G:N|$ by \cite[Cor.~11.29]{Is}.

\begin{conj}   \label{conj:aqs}
 Let $G$ be a finite group such that $S \lhd G/\bZ(G)\leq\Aut(S)$ for some
 non-abelian simple group $S$, and let $\chi\in\Irr(G)$ be faithful. Suppose
 that 
 \begin{enumerate}[\rm(1)]
  \item $\chi$ is irreducible over $G^{(\infty)}$, and
  \item $\chi(g) \neq 0$ and $G=\langle G^{(\infty)}, \bZ(G), g\rangle$ 
   for some $g \in G$. 
 \end{enumerate}
 Then $o(g\bZ(G))$ divides $|G:\bZ(G)| /\chi(1)$.
\end{conj}

\begin{thm}   \label{thm:reduction}
 Assume Conjecture~\ref{conj:aqs} holds.
 Let $G$ be a finite group, and $\chi\in\Irr(G)$. Let $N\nor G$,
 $\theta\in \Irr(N)$, and $\chi\in\irr{G \mid \theta}$. Let $g \in G$.
 If $\chi(gn)\ne 0$ for some $n\in N$, then
 $$o(gN)\quad\text{divides}\quad \frac{|G:N|\theta(1)}{\chi(1)}\, .$$
\end{thm}

\begin{proof}
We proceed by induction on $|G:N|$.
\smallskip

(1) Suppose that $N\le U < G$ and $\psi\in\Irr(U)$ is such that $\psi^G=\chi$.
Since $\chi(gn)\ne0$, by the induction formula there is $x\in G$ such that
$(gn)^x=g^xn^x \in U$ and $\psi((gn)^x)\ne 0$. Now, $N\nor U$, $g^x\in U$ and
by induction, we have that $o(gN)$ divides
$$\frac{|U:N|\theta(1)}{\psi(1)}=\frac{|G:N|\theta(1)}{\chi(1)}\, .$$
In particular, we may assume that $\theta$ is $G$-invariant, or in other words,
that $(G,N, \theta)$ is a character triple. By \cite[Thm~8.2]{Is73} and
\cite[Lemma 5.17(a)]{Na18} applied to $x=gn$, we may work in a character
triple, say $(G,N, \theta)$ such that $N \leq \bZ(G)$, $\theta$ is linear and
faithful, and each coset of $N$ contains some $y$ such that
$N\cap\langle y\rangle=1$.
Since $N$ acts by scalars, we have $\chi(gn)\ne 0$ for all $n\in N$. Now, given
the coset $Ng$, we know that there is $y\in Ng$ such that
$\langle y\rangle \cap N=1$.
Hence, $o(yN)=o(y)$, $\chi(y)\ne 0$ and we may assume that $y=g$. 
\smallskip

(2) If $J=\ker \chi$, then $J\cap N=1$. Assume that $J>1$. Then 
$$|G/J:JN/J|<|G:N|,$$ 
and by induction $o(gJ)$ divides $|G:JN|/\chi(1)$. Since $o(g)$ divides
$|J|o(gJ)$, then we are done. So we may assume that $\chi$ is faithful.

Since $N$ is central by (1), we may assume that $\chi$ is primitive.

Suppose that $G$ is solvable. Let $A=N\langle g\rangle$, which is abelian.
By \cite[Cor.~4.2]{Wi06} we know that $\chi(1)$ divides $|G:A|$,
and in this case we are done. 
\smallskip

(3) We may now assume that $G/N$ is non-solvable.
Let $K/N$ be a chief factor of~$G$. Write $\chi_K= e \eta$, where
$\eta\in\Irr(K)$ and $e\ge1$. Let $H=K\langle g\rangle$.
Since $\chi(g)\ne 0$, there exists an irreducible constituent $\nu\in\Irr(H)$
of $\chi_H$ such that $\nu(g)\ne 0$. We have $\nu_K=\eta$ since $H/K$ is
cyclic. Since $|G:K|<|G:N|$, we have that $o(gK)$ divides 
$$\frac{|G:K|\eta(1)}{\chi(1)}\, .$$

Assume that $H<G$. Then by induction, we have that $o(gN)$ divides 
$$\frac{|H:N|}{\eta(1)}=\frac{|H:K||K:N|}{\eta(1)}
  =o(gK)\frac{|K:N|}{\eta(1)} $$
which in turn divides
$$\frac{|G:K|\eta(1)}{\chi(1)} \frac{|K:N|}{\eta(1)}=\frac{|G:N|}{\chi(1)},$$ 
as wanted. So we may assume that $G=H$, and thus $G = K\langle g \rangle$,
whence $K/N$ is non-solvable.

Recall from (1) that $N \leq \bZ(G)$.
If $N < \bZ(G)$, then, by the induction hypothesis applied to $G/\bZ(G)$,
$a:=o(g\bZ(G))$ divides $|G/\bZ(G)|/\chi(1)$. Now $g^a \in \bZ(G)$, and so
$g^{ab} \in N$ for $b:=|\bZ(G)/N|$. Hence $o(gN)$ divides $ab$, which divides
the integer $b\cdot |G/\bZ(G)|/\chi(1) = |G/N|/\chi(1)$, and so we are done. 
Therefore we may assume that 
\begin{equation}\label{eq:cent1}
  N = \bZ(G).
\end{equation}
\smallskip

(4) Write $K/N=(K_1/N) \times \cdots \times (K_{n}/N)$, where $K_i/N \cong S$
is simple non-abelian. Then $K_i = NL_i$ with $L_i:=K_i^{(\infty)}$ being a
quasi-simple cover of $S$.
When $i \neq j$, we have $[L_i,L_j] \leq [K_i,K_j] \leq N$, whence the Three
Subgroups Lemma and \eqref{eq:cent1} imply
$$[L_i,L_j] = [[L_i,L_i],L_j] \leq [L_i,N] = 1.$$
It follows that $[K_i,K_j]= [NL_i,NL_j] = 1$ when $i \neq j$,
and that $K$ is the central product of $K_1, \ldots, K_{n}$.
Without loss, we may assume that the conjugation by $g$ acts as follows
\begin{equation}\label{for-g1}
  K_1 \mapsto K_2 \mapsto \ldots \mapsto K_n \mapsto K_1.
\end{equation}
Using $\chi(g) \neq 0$, or primitivity of $\chi$, we have that $\chi|_K$ is
irreducible. Now, $K$, being a central product of $K_1,\ldots,K_n$, is a
quotient of $\hat{K}:= K_1\times\cdots\times K_n$ (by a central subgroup).
Hence, if $\chi$ is afforded by a representation $\Phi:G \to \GL(V)$, then we
can realise the $K$-module $V$ as the $\hat{K}$-module 
$$V_1 \otimes V_2 \otimes \cdots \otimes V_n,$$
where each $K_i$ acts irreducibly on $V_i$ and trivially on $V_j$ when $j\neq i$
and moreover the~$V_i$ are permuted transitively by $g$, so have same dimension.
Then it is known (see e.g. \cite[Lemma 2.6]{GT}) that 
$$\Phi(G) \le \bigl( \bigotimes^n_{i=1}\GL(V_i) \bigr) \rtimes \fS_n;$$
moreover, \eqref{for-g1} implies that $g^n$ maps $V_1$ to $V_1$.
Fix a basis $(e_1, \ldots,e_d)$ of $V_1$, and let~$h$ denote the matrix
of~$g^n$ with respect to this basis. Now we consider the basis 
$$\{e^i_j = g^{i-1}(e_j)\mid 1 \leq j \leq d\}$$
for each $V_i$ when $1 \leq i \leq n$. With respect to this basis, $g$ acts as
$\si(1\otimes \cdots \otimes 1\otimes h)$ where $\si$ is the $n$-cycle
$(1,2, \ldots,n)$; more precisely,
$$g: v_1 \otimes v_2 \otimes \cdots \otimes v_n
  \mapsto h(v_n) \otimes v_1 \otimes v_2 \otimes \cdots \otimes v_{n-1}$$
for $v_i \in V_i$. In particular,
$$g(e^1_{i_1} \otimes e^2_{i_2} \otimes \cdots \otimes e^n_{i_n})
  \mapsto h(e^1_{i_n}) \otimes e^2_{i_1} \otimes e^3_{i_2} \otimes 
     \cdots \otimes e^{n}_{i_{n-1}}.$$
The contribution of this basis vector $e^1_{i_1}\otimes\cdots\otimes e^n_{i_n}$
to $\chi(g) = \Tr(\Phi(g))$ can be non-zero only when $i_1 = \ldots = i_n$, in
which case it is equal to the coefficient of $e^1_{i_1}$ in $h(e^1_{i_1})$.
Thus we have proved the following formula (mentioned without proof on
\cite[p.~515]{GI})
\begin{equation}   \label{for-g2}
  \chi(g) = \Tr(h).
\end{equation}
\smallskip

(5) Consider the subgroup
\begin{equation}   \label{for-m1}
  M_1 := \langle K_1,g^n \rangle
\end{equation}
and note that $M_1$ acts on $V_1$, with $g^n$ acting via $h$. Let $Z_1$ consist
of the elements of~$M_1$ that act on $V_1$ as scalars. Recall that $K_1=L_1N$
acts irreducibly on $V_1$, $N = \bZ(G)$ acts by scalars on~$V_1$, and
$L_1 = K_1^{(\infty)}$. Hence $L_1 = M_1^{(\infty)}$ and $L_1$ acts irreducibly
on~$V_1$. Now, $\bC_{M_1}(L_1)$ acts by scalars on $V_1$, and so 
$\bC_{M_1}(L_1) = Z_1$. It follows that 
$$S \cong L_1/\bZ(L_1)= L_1/(L_1 \cap Z_1) \cong L_1Z_1/Z_1 \leq M_1/Z_1
  \leq \Aut(L_1) \leq \Aut(S).$$
We also have that $M_1=\langle M_1^{(\infty)},N,g^n\rangle$. By \eqref{for-g2},
$\Tr(h) \neq 0$. Hence, Conjecture~\ref{conj:aqs} applied to the
$M_1/Z'_{1}$-module $V_1$, where $Z'_{1}$ denotes the kernel of $M_1$ acting
on $V_1$, implies that 
\begin{equation}   \label{for-g3}
  s:=o(hZ_1) \mbox{ divides }|M_1/Z_1|/d.
\end{equation}

By its definition, $h^s \in \GL(V_1)$ is a scalar matrix, and thus $g^{ns}$
acts as a scalar say $\gamma$ on~$V_1$. But $V_i = g^{i-1}(V_1)$ and
$[g^{ns},g^{i-1}]=1$, so $g^{ns}$ also acts as the same scalar $\gamma$ on
each $V_i$, $1\leq i \leq n$. Thus $\Phi(g^{ns})\in \GL(V)$ is a scalar matrix.
As $\Phi$ is faithful, it follows that $g^{ns} \in \bZ(G)=N$
and thus $o(gN)$ divides $ns$. It then follows from \eqref{for-g3} that
\begin{equation}   \label{for-g4}
  o(gN) \mbox{ divides }n|M_1/Z_1|/d.
\end{equation}
Recall that $G=K\langle g \rangle$, and $K$ normalises each $K_i$. Hence
\eqref{for-g1} implies that 
\begin{equation}   \label{for-g5}
  |G:\langle K,g^n \rangle|=n.
\end{equation}
Next, using \eqref{for-m1} we have 
$$\langle K,g^n \rangle = K_1K_2 \cdots K_n \langle g^n \rangle = RM_1,$$
for $R := K_2K_3 \cdots K_n$. Now $R \cap M_1$ centralises $K_1$, and $K_1$ is
irreducible on $V_1$, so $R \cap M_1$ acts on $V_1$ by scalars, and thus
$R \cap M_1 \leq Z_1$. Letting $c:=|Z_1/(R \cap M_1)|$, we obtain
$$|\langle K,g^n \rangle/N| = \frac{|RM_1|}{|N|}
  = \frac{|R| \cdot |M_1|}{|R \cap M_1| \cdot |N|} = c|R/N| \cdot |M_1/Z_1|
  = c|S|^{n-1} \cdot |M_1/Z_1|.$$
Together with \eqref{for-g5} and using $\chi(1)=d^n$, we now have
$$|G/N|/\chi(1) = n|\langle K,g^n \rangle/N|/d^n
    = c\frac{|S|^{n-1}}{d^{n-1}} \cdot n\frac{|M_1/Z_1|}{d}.$$
As $K_1/N \cong S$ and $K_1$ acts irreducibly on $V_1$, $d=\dim V_1$
divides~$|S|$. Hence $|G/N|/\chi(1)$ is a multiple of $n|M_1/Z_1|/d$, and so is
divisible by $o(gN)$ by \eqref{for-g4}.
\end{proof}

In fact, it is more convenient to work with the following equivalent version
of Conjecture~\ref{conj:aqs} (which is precisely Condition~\eqref{condB} for $(H,h,\chi)$ satisfying 
$(*)$): 

\begin{conj}   \label{conj:aqs12}
 Let $H$ be a finite group where $L:=[H,H]$ is quasi-simple with
 $\bZ(L)=\bZ(H)$, and let $\chi\in\Irr(H)$ be faithful. Then for any $h\in H$
 with $\chi(h)\neq 0$ and $H=\langle L, h\rangle$ we have that $o(h\bZ(H))$
 divides $|H:\bZ(H)| /\chi(1)$.
\end{conj}

\begin{lem}\label{2aqs}
 Conjecture~{\rm\ref{conj:aqs}} is equivalent to
 Conjecture~{\rm\ref{conj:aqs12}}.
\end{lem}

\begin{proof}
Certainly, $(H,h,\chi)$ as in Conjecture~\ref{conj:aqs12} satisfy the
assumptions of Conjecture~\ref{conj:aqs}, with $G=H$, $G^{(\infty)}=L$, $g=h$.
(Note that $\chi(h) \neq 0$ implies that $\chi|_L$ is irreducible.
Indeed, since $H = \langle L,h \rangle$, $\chi(h)\neq 0$ implies 
that an irreducible constituent $\theta$ of $\chi_L$ is $H$-invariant.
But $H/L$ is cyclic, so $\theta$ is extendible to $H$, and $\theta=\chi_L$.)

Conversely, consider any $(G,g,\chi)$ that satisfies the assumptions of
Conjecture~\ref{conj:aqs}, and let $\chi$ be afforded by a (faithful)
representation $\Phi:G \to \GL(V)$. We will identify~$G$ with $\Phi(G)$, and
create a new subgroup $H <\GL(V)$ that satisfies the assumptions set in
Conjecture~\ref{conj:aqs12}.

Since $\Phi|_L$ is irreducible for $L:=G^{(\infty)}$, $\bC_G(L) = \bZ(G) =:Z$,
$\bZ(L) = Z \cap L$, whence
\begin{equation}\label{eq20a}
  S \cong L/\bZ(L) \cong LZ/Z \leq G/Z \leq \Aut(S). 
\end{equation}
Let $n=ab$ denote the order of $gZ$ in $G/Z$, where $a$ is the order of $gLZ$
in $G/LZ \cong (G/Z)/(LZ/Z)$, and $b$ is the order of $g^aZ$ in
$LZ/Z \cong L/\bZ(L)$. Then we can write 
$$g^a = tz,$$ 
where $t\in L$ and $z = \zeta\cdot\Id$ (recall we have identified $G$ with the
subgroup $\Phi(G)$ of $\GL(V)$). 
Choose $\gamma \in \CC^\times$ such that $\gamma^a=\zeta$, and define
$$h:= \gamma^{-1} g \in \GL(V),~H := \langle L, h \rangle.$$
Then $h^a = t \in L$, so $H$ is finite. Furthermore, $L$ is quasi-simple
by~\eqref{eq20a}, whence $[H,H]=L$. 
Since $L < \GL(V)$ is irreducible, we again have that 
\begin{equation}\label{eq20}
  \bC_H(L) = \bZ(H),~\bZ(L) = \bZ(H) \cap L.
\end{equation}
As conjugations by $h$ and $g$ induce the same automorphism of $L$, we see that
\begin{equation}\label{eq21}
  o(h\bZ(H)) = o(gZ) = ab.
\end{equation}
Next, we claim that 
\begin{equation}\label{eq22}
  |H|=a|L|.
\end{equation}
Indeed, $h^a\in L$. Suppose $h^c\in L$ for some divisor $c< a$ of $a$. Then, as
$g^c = (\gamma h)^c = \gamma^c h^c$, we have $G\ni g^ch^{-c}=\gamma^c\cdot \Id$,
so $g^ch^{-c} \in Z=\bZ(G)$ and thus $g^c \in LZ$, contrary to the fact that
$a$ is the order of $gLZ$ in $G/LZ$.

Now, $G = \langle L,Z,g\rangle$ induces by conjugation a subgroup of order
$a|S|$ of $\Aut(L)$, which is isomorphic to $G/\bZ(G)$. Since $H$ and $G$ 
are the same modulo $\bZ(\GL(V))$, $H$ also induces the same subgroup of order
$a|S|$ of $\Aut(L)$, which is now isomorphic to $H/\bZ(H)$. Recalling
\eqref{eq20} and~\eqref{eq22}, we then have
$$a|L|/|\bZ(H)| = |H/\bZ(H)| = |H/\bC_H(L)| = a|S| = a|L|/|\bZ(L)|,$$
and so $\bZ(H)=\bZ(L)$. Now $\Tr(h) = \gamma^{-1}\chi(g) \neq 0$, so by the
assumption on $(H,h)$ we must have that $o(h\bZ(H))$ divides
$|H/\bZ(H)|/\dim(V) = |G/\bZ(G)|/\chi(1)$. By~\eqref{eq21},
then $o(gZ)$ divides $|G/\bZ(G)|/\chi(1)$, as desired. 
\end{proof}

\begin{proof}[Proof of Theorem B]
Suppose Conjecture~\ref{conj:aqs12} holds for all nearly simple groups $H$
satisfying Condition~$(*)$. By Lemma \ref{2aqs}, Conjecture~\ref{conj:aqs}
holds for all finite non-abelian simple groups $S$. Applying
Theorem~\ref{thm:reduction} with $N=1$, we see that Conjecture A holds.
\end{proof}

For groups $G$ corresponding to a given triple $(G/\bZ(G),G^{(\infty)},\chi)$
as in Conjecture~\ref{conj:aqs}, Conjecture~\ref{conj:aqs12} allows us to work
with groups of minimal order. In general, for any such triple, it suffices to
establish Conjecture~\ref{conj:aqs} for one particular group $G$, as shown in
the following statement:

\begin{lem}   \label{2aqs2}
 Suppose $G$ and $H$ are finite groups with $G/\bZ(G) \cong H/\bZ(H)$ and
 $G^{(\infty)} = H^{(\infty)}=: L$. 
 Suppose a faithful $\chi\in\Irr(G)$ and a faithful $\psi\in\Irr(H)$ satisfy 
 $\chi_L = \psi_L\in\Irr(L)$. Suppose $g\in G$ and $h\in H$ are such that 
 $G=\langle L,\bZ(G),g\rangle$, $H=\langle L,\bZ(H),h\rangle$,
 $\chi(g)\neq 0 \neq \chi(h)$, and moreover the conjugations by~$g$ and by~$h$
 induce the same automorphism of $L$. If Conjecture~{\rm\ref{conj:aqs}} holds
 for $(G,g,\chi)$, then Conjecture~{\rm\ref{conj:aqs}} also holds for
 $(H,h,\psi)$.
\end{lem}
 
\begin{proof}
Assume Conjecture~{\rm\ref{conj:aqs}} holds for $(G,g,\chi)$, and let $\chi$ be
afforded by a faithful representation $\Phi:G \to\GL(V)\cong\GL_n(\CC)$.
As $\chi_L=\psi_L$, we may assume that $\psi$ is afforded by a faithful
representation $\Psi:H\to\GL(V)$, where $\Phi_L=\Psi_L$ is irreducible. 
Using faithfulness we can identify $G$ with $\Phi(G)$ and $H$ with $\Psi(H)$.
Now, since $g$ and~$h$ induce the same automorphism of $L$, by Schur's lemma
we have $h=cg$ for some $c\in\CC^\times$. As $\bZ(G)$ and $\bZ(H)$ consist of
scalars, it follows that $o(g\bZ(G)) = o(h\bZ(H))$. Finally, since
$|G:\bZ(G)|/\chi(1)=|H:\bZ(H)|/\psi(1)$, the statement follows.
\end{proof}

\section{Verifying Conjecture~\ref{conj:aqs12} at a prime}   \label{sec:oneprime}

Observe that the claim in Conjecture~\ref{conj:aqs12} can be checked
individually for the various prime divisors of $o(h)$. Recall the conditions
$(\ddagger)$ and $(\ddagger^\star)$ on a finite group $G$, a prime $p$ and a
character $\chi\in\Irr(G)$ discussed in the introduction:
$$\text{defect groups $D$ of the $p$-block of $\chi$ satisfy
$\exp(D) \le \left(\frac{|G|}{\chi(1)}\right)_p=p^{\df(\chi)},$}\eqno{(\ddagger)}$$
$$\text{defect groups $D$ of the $p$-block of $\chi$ satisfy
$\exp(D/\bZ(G)_p) \le \left(\frac{|G:\bZ(G)|}{\chi(1)}\right)_p.$}\eqno{(\ddagger^\star)}$$
We say that a $p$-block $B$ \emph{satisfies $(\ddagger)$, respectively
$(\ddagger^\star)$} if all $\chi\in\Irr(B)$ do satisfy $(\ddagger)$,
respectively $(\ddagger^\star)$. Note that  these two conditions are equivalent
if $p \nmid |\bZ(G)|$.

Using the validity of the conjecture of Robinson \cite[Thm~3]{KM25} yields the
following criterion:

\begin{cor}   \label{cor:rob}
 Suppose $B$ is a block with defect group $D$ for which $|\bZ(D)|\ge \exp(D)$.
 Then $(\ddagger)$ holds for $B$.
\end{cor}

The following observation will allow us to descend to quasi-simple groups in
certain situations:

\begin{prop}   \label{prop:to qs}
 Let $H$ be as in Conjecture~{\rm\ref{conj:aqs12}}. If all $p$-blocks of
 $L=[H,H]$ satisfy $(\ddagger^\star)$, then $(\ddagger^\star)$ holds at $p$ for
 all $\chi\in\Irr(H)$ with irreducible restriction to $L$, and thus
 Conjecture~{\rm\ref{conj:aqs12}} holds for $H$ at~$p$.
\end{prop}

\begin{proof}
Let $\chi\in\Irr(H)$ with $\chi|_L$ irreducible. Now if $h\in H$ with
$\chi(h)\ne0$ then $h_p$ lies in a defect group $D$ of $\chi$, by
\cite[9.26 and Thm~5.9]{Na98}, hence 
$$o(h\bZ(H))_p=o(h_p\bZ(H)))\le \exp(D/\bZ(H)_p).$$
Let $D_0\le D$ be a defect group of $\chi|_L$, then $|D:D_0|\le |H:L|_p$. By
assumption we have $\exp(D_0/\bZ(L)_p)\le |L:\bZ(L)|_p/\chi_L(1)_p$ and
$\bZ(L)=\bZ(H)$, so we obtain
$$o(h\bZ(H))_p\le \exp(D/\bZ(H)_p)\le |H:L|_p\,\exp(D_0/\bZ(L)_p)
  \le |H:\bZ(H)|_p/\chi(1)_p.$$
To conclude, as we noted in the proof of Lemma~\ref{2aqs}, in the situation of
Conjecture~{\rm\ref{conj:aqs12}}, $\chi(h) \neq 0$ implies $\chi|_L$ is
irreducible. 
\end{proof}

Using the validity of one direction of Brauer's height zero conjecture \cite{KM}
this yields the following powerful criterion:

\begin{cor}   \label{cor:ab}
 In the situation of Conjecture~{\rm\ref{conj:aqs12}}, assume $\chi|_L$ lies
 in a $p$-block $B$ with abelian defect. Then Conjecture~{\rm\ref{conj:aqs12}}
 holds for $\chi$ at $p$.
\end{cor}

\begin{proof}
By \cite[Thm~1.1]{KM}, Brauer's height zero conjecture holds for $B$ (with
defect group $D$), so all characters $\chi$ in $B$ have height~0. Thus
$|L|_p/\chi_L(1)_p=p^{\df(\chi|_L)} =|D|$, and 
so 
$$|L:\bZ(L)|_p/\chi_L(1)_p=|D/\bZ(L)_p|\ge\exp(D/\bZ(L)_p).$$ 
Now apply Proposition~\ref{prop:to qs} to conclude.
\end{proof}

We will make use of the following easy observation:

\begin{lem}   \label{lem:updown}
 Let $N\unlhd G$ be finite groups and let $B$ be a $p$-block of $G$.
 \begin{enumerate}[\rm(a)]
  \item Assume $|G:N|$ is prime to $p$. Then $(\ddagger)$ {\rm (}resp.\ 
   $(\ddagger^\star)${\rm )} holds for $B$ if and only if it holds for some
   block of $N$ covered by $B$.
  \item Assume $|N|$ is prime to $p$. Let $\bar B$ be a block of $G/N$
   and assume $B$ is the unique block of $G$ such that $\bar B \subseteq B$.
   Then $(\ddagger)$ {\rm (}resp.\ $(\ddagger^\star)${\rm )} holds for $B$ if
   and only if it holds for $\bar B$.
 \end{enumerate}
\end{lem}

\begin{proof}
For (a) let $b$ be an $p$-block of $N$ covered by $B$. By
\cite[Thm~9.26]{Na98}, $b$ and $B$ have a common defect group $D$.
If $\chi\in\Irr(G)$ and $\theta\in\Irr(N)$ lies under $\chi$, then
$\chi(1)_p=\theta(1)_p$ by \cite[Cor.~11.29]{Is}. Also, $\bZ(G)_p\le N$ by
assumption. Now, the proof of (a) is complete using that every character of~$b$
lies under some character of~$B$ \cite[Thm~9.4]{Na98} and every character
of~$B$ lies over some character of~$b$
\cite[Thm~9.2]{Na98}.   \par
In (b), we know that $\Irr(B)=\Irr(\bar B)$ by \cite[Thm~9.9(c)]{Na98}. 
Also, if $D$ is a defect group of $B$, then $DN/N$ is a defect group of
$\bar B$. Finally, $\exp(DN/N\bZ(G)_p)=\exp(D/\bZ(G)_p)$ as $N\cap\bZ(G)_p=1$.
The claim easily follows.
\end{proof}

Next we present a useful general result on blocks which perhaps is not as
well--known as it should be. The proof below is due to G.~R.~Robinson, whom
we thank for this.

\begin{thm}   \label{thm:geoff}
 Let $B$ be a $p$-block with defect group $D$ of a finite group $G$.
 Then a $p$-element $x \in G$ lies in some $G$-conjugate of $D$ if and only if
 $\chi(x)\ne 0$ for some $\chi\in\irr B$.
\end{thm}

\begin{proof}
Suppose that $\chi(x)\ne 0$ for some $\chi\in\irr B$. Then $x$ lies in some
$G$-conjugate of $D$ by \cite[Cor~5.9]{Na98}.  Conversely, assume that
$x \in D$ and that $\chi(x)=0$ for all $\chi\in\irr B$. Then we have
$$0=\sum_{\chi\in\irr B} |\chi(x)|^2
   = \sum_{\chi\in\irr B} \overline{\chi(x)}\chi(x) \, .$$
By \cite[Cor~5.8]{Na98}, with the generalised decomposition numbers
$d^x_{\chi \mu}$ we have
$$\chi(x)=\sum_{\mu \in \Delta} d^{x}_{\chi \mu}\, \mu(1)\, ,$$
where $\Delta:=\{\mu\in\ibr b\mid \text{$b$ a block of $\cent Gx$
inducing $B$}\}$. Hence
$$\begin{aligned}
0=&\sum_{\chi\in\irr B}\left(\sum_{\mu\in\Delta}\overline{d^{x}_{\chi\mu}} \mu(1)\right)\left(\sum_{\vhi\in\Delta} d^x_{\chi\vhi} \vhi(1)\right)\\
 =&\sum_{\mu,\vhi\in\Delta} \left(\sum_{\chi\in\irr B}\overline{d^{x}_{\chi \mu}}d^{x}_{\chi\vhi}\right)\mu(1)\vhi(1)
  =\sum_{\mu, \vhi \in \Delta} c_{\mu \vhi}\mu(1)\vhi(1)\, ,\end{aligned}$$
by using \cite[Lemma~5.13(b)]{Na98}. Since the Cartan numbers
$c_{\mu \vhi}\ge 0$ and $c_{\mu \mu}>0$, we conclude that there are no
$p$-blocks $b$ of $\cent Gx$ which induce $B$. Now let $\psi\in\irr B$ of
height zero. By \cite[Thm~5.14]{Na98}, there is a $p$-regular $y\in\cent Gx$
such that $\psi(xy)\ne 0$.  By \cite[Cor.~5.8]{Na98}, we conclude that
necessarily there exists a $p$-block of $\cent Gx$ that induces $B$.
This is a final contradiction.
\end{proof}

Unfortunately, for no prime $p$ does $(\ddagger)$ hold for all $p$-blocks of
all finite groups:

\begin{exmp}   \label{gunterexmp}
 Let $p$ be a prime, $a\ge1$ and $q$ a prime power such that $p^a$ divides
 $q-1$ (by Dirichlet, such $q$ exist for all $a\ge1$). Then the principal
 $p$-block of $G=\SL_p(q)$ contains an irreducible Deligne--Lusztig character
 $\chi$ parametrised by a $p$-element in the dual group $\PGL_p(q)$ with
 centraliser order $p(q^p-1)/(q-1)$. So $\chi$ has defect~2, while the exponent
 of a Sylow $p$-subgroup of $G$ is at least $p^a$. Thus, in general, the
 exponent of a defect group can not be bounded by \emph{any} function of the
 minimal defect in the block.
\end{exmp}

\begin{prop}   \label{prop:spor}
 Conjecture~{\rm\ref{conj:aqs12}} holds whenever the non-abelian simple
 composition factor of $H$ is either a sporadic simple group, the Tits group,
 the group $\fA_6$ or~$\fA_7$, or a simple group of Lie type with an exceptional
 Schur multiplier. Moreover, in these cases, Conditions~$(\ddagger)$ and
 $(\ddagger^\star)$ are satisfied for all $p$-blocks of $[H,H]$ whenever
 $p\ge3$.
\end{prop}

\begin{proof}
The character tables of all quasi-simple groups as in
the claim are contained in {\sf GAP} \cite{GAP}. From these, using
Theorem~\ref{thm:geoff} to control the exponent of defect groups, it is easily
checked that~$(\ddagger)$ holds for all primes $p\ge3$, as well as for all
2-blocks except possibly for certain blocks of defect~4 and maximal height~2,
of
$$M_{11},2.Suz,HN,2.\fA_6,6.\fA_6,2.\fA_7,6.\fA_7,2.\OO_7(3),6.\OO_7(3).$$
For these remaining groups $L$, all cyclic extensions $H$ with $\bZ(H)=\bZ(L)$
are again available in {\sf GAP} and thus the assertion of
Conjecture~\ref{conj:aqs12} can be checked directly.
\end{proof}

We now treat nearly simple groups with simple composition factor an
alternating group. Note that the case of the symmetric group is mentioned
in the introduction to Wilde's paper \cite{Wi06}. We refer to \cite{Ol} for
properties of blocks of symmetric and alternating groups and their covering
groups.

\begin{prop}   \label{prop:Sn}
 Let $p$ be a prime. Then any $p$-block of $\fS_n$ and $\fA_n$
 satisfies~$(\ddagger)$.
\end{prop}

\begin{proof}
The $p$-blocks of $\fS_n$ are parametrised by $p$-cores of partitions of $n$.
That is, if $B$ is a $p$-block of $\fS_n$ then there exists a $p$-core $\mu$
such that the irreducible character~$\chi^\la$ of $\fS_n$ labelled by the
partition $\la\vdash n$ lies in $B$ if and only if $\mu$ is the $p$-core
of~$\la$. The weight $w(B):=(n-|\mu|)/p$ is the number of $p$-hooks that have
to be removed from $\la$ to reach its $p$-core~$\mu$. A defect group $D$ of $B$
is then isomorphic to Sylow $p$-subgroups of $\fS_{pw}$. Clearly we
may assume $w>0$. Thus, if
$w=\sum_{i=0}^ra_i p^i$ is the $p$-adic expansion of~$w$, with $0\le a_i<p$ and
$a_r>0$, then $D\cong \prod_{i=0}^r(C_p\wr\cdots\wr C_p)^{a_i}$, where in the
$i$th term we have iterated wreath products of $i+1$ factors. Thus clearly the
exponent of $D$ equals $p^{r+1}$, and since $r\le\log_p w$ we get
$\exp(D)\le pw$.   \par
On the other hand, any partition $\la$ with $p$-core $\mu$ has at least $w$
hooks of length divisible by $p$. Thus, by the hook length formula for the
character degree $\chi^\la(1)$ we obtain
$$\frac{|\fS_n|_p}{\chi^\la(1)_p}\ge p^w \ge pw\ge \exp(D),$$
as desired.   \par
For $\fA_n$ the claim follows directly from the $\fS_n$-result by
Lemma~\ref{lem:updown}(a) if $p$ is odd. If $p=2$, let $B'$ be a $2$-block of
$\fA_n$ and $B$ a $2$-block of $\fS_n$ covering it. If $\chi\in\Irr(B)$ lies
above $\chi'\in\Irr(B')$ then the degrees $\chi'(1)$ and $\chi(1)$ differ by at
most a factor of~2. The defect groups of $B'$ are non-abelian if and only
if the weight of $B$ satisfies $w\ge3$. If $w=3$ then $\exp(D)=4$ while if
$w\ge4$ then $p^w\ge 2pw$, so again the above inequality gives the desired
result. If the defect groups are abelian, we may apply Corollary~\ref{cor:ab}.
\end{proof}

\begin{prop}   \label{prop:2Sn odd}
 Let $p$ be an odd prime. Then any $p$-block of $2.\fS_n$ and of $2.\fA_n$
 satisfies~$(\ddagger)$.
\end{prop}

\begin{proof}
The $p$-blocks of any double cover $2.\fS_n$ of $\fS_n$ are parametrised by
$p$-bar cores of (certain) partitions of $n$. The weight of a block is then
defined as for $\fS_n$ and a defect group $D$ of a block of weight $w$ is
isomorphic to those of a block for $\fS_n$ of weight~$w$, whence we again have
$\exp(D)\le pw$. Then as for $\fS_n$ the bar hook length formula for the
character degrees of $2.\fS_n$ yields Condition~$(\ddagger)$. The descent to
$2.\fA_n$ is immediate by Lemma~\ref{lem:updown}(a).
\end{proof}

\begin{thm}   \label{thm:An}
 Conjecture~{\rm\ref{conj:aqs12}} holds for $L/\bZ(L)=\fA_n$, $n\ge5$.
\end{thm}

\begin{proof}
By Proposition~\ref{prop:spor} we may assume $n\ne 6,7$, and so $2.\fA_n$ is
the unique non-trivial covering group of $\fA_n$. If $p>2$ or $L=\fA_n$ the
claim follows by Proposition~\ref{prop:to qs} in conjunction with
Propositions~\ref{prop:Sn} and~\ref{prop:2Sn odd}. So now assume $p=2$ and
$L=2.\fA_n$. The faithful characters of $G=\hat\fS_n:=2.\fS_n$ are labelled by
partitions $\la$ of $n$ with distinct parts (in a not necessarily unique way),
and as customary, we write $\sp\la\in\Irr(\hat\fS_n)$ for such a character.
By \cite[Thm~8.7]{HH92} if $\sp\la$ takes non-zero value on $g\in\hat\fS_n$
then either $g$ projects to an element of odd order in $\fS_n$, or to an
element with cycle type~$\la$. In the first case $o(g\bZ(G))_2=1$, so we are
done. Now assume that $g$ has cycle type~$\la$.
To treat this situation we seem to need some more detailed knowledge on
2-blocks of~$\hat\fS_n$. So assume that $\sp\la$ lies in a 2-block $B$ of
$\hat\fS_n$ of weight~$w$. The height of spin characters in~$B$ is bounded
above by $3w/2-l$ by \cite[Thm~3.8]{BO97}, where $l$ is the number of digits
in the 2-adic expansion of $w$. Since defect groups of $B$ are isomorphic to
Sylow 2-subgroups of $\hat\fS_{2w}$ by \cite[Thm~3.2]{BO97}, they have order
$2^{2w-l}$, so the minimal defect in $B$ is at least $w/2$. On the other
hand, the exponent of a Sylow 2-subgroup of $\fS_{2w}$ is the largest 2-power
not bigger than~$2w$. Comparing we find that our claim holds unless possibly
when $w=2$ or $w=4$. Now for $w=4$ the largest height is in fact~4 (see
\cite[Exmp.~3.11]{BO97}), so the smallest defect is~3, and we are done again.
\par
Finally assume $w=2$. By the procedure described in \cite[\S3]{BO97} for
finding the $\bar 4$-abacus, if $\sp\la$ is a spin character in a 2-block of
weight~2 then $\la$ has either only odd parts, or exactly one even part, of
size~4. Now if $\la$ has only odd parts, then by \cite[Thm~8.7]{BO97},
$\sp\la$ vanishes on elements of order divisible by~4, while if $\la$ has a
part~4, then the defect of $\sp\la$ is~8 and our desired inequality holds.
\end{proof}

\section{The groups of Lie type}   \label{sec:Lie}
In this section we prove that most blocks of quasi-simple groups of Lie
type satisfy $(\ddagger)$, and we verify Condition~\eqref{condB} for a few
of the remaining cases. The investigation naturally splits according to whether
the considered prime equals the underlying characteristic of our group, or not.
We refer to \cite{MT} and \cite{GM20} for notation and background on the
structure and representation theory of these groups.

\subsection{Groups of Lie type in their defining characteristic}

Let $\bG$ be a simple algebraic group of simply connected type over a field
of characteristic $p>0$, with a Steinberg endomorphism $F:\bG\to\bG$, and
$G=\bG^F$ the corresponding finite group of Lie type.

\begin{prop}   \label{prop:defchar}
 Condition~$(\ddagger)$ holds for any $p$-block of $G=\bG^F$ in its defining
 characteristic~$p$.
\end{prop}

\begin{table}[htb]
\caption{Bounds for $p$-parts in character degrees of groups of Lie type}   \label{tab:Achi}
$$\begin{array}{c|ccccccccccccc}
  \bG & A_n& B_n& C_n& D_n& E_6& E_7& E_8& F_4& G_2\\
  & & n\ge3& & n\ge4&\\
\hline
 A_\chi\ge& n& 2n-1& 2n-1& 2n-3& 11& 17& 29& 11& 5\\
 N-N_1\ge& n& 2n-2& 2n-2& 2n-2& 16& 26& 56& 8& 3\\
 \log_p(\exp(S))\le& \log_pn& \log_p(2n+1)& \log_p(2n)& \log_p(2n)& 4& 5& 5& 4& 3\\
\end{array}$$
\end{table}

\begin{proof}
The group $G$ has one $p$-block of defect zero, containing the Steinberg
character, while all other $p$-blocks of $G$ are of maximal defect (see
\cite[Thm~6.18]{CE}). Let $P$ be a Sylow $p$-subgroup of $G$, of order $q^N$,
where $q$ is the absolute value of all eigenvalues of $F$ on the character
group of an $F$-stable maximal torus of $\bG$, and $N$ is the number of
positive roots of $\bG$. The degrees of irreducible characters of~$G$ are given
by degree polynomials evaluated at the underlying field size $q$ (see
\cite[Def.~2.3.25]{GM20}). The maximal power of $q$ dividing the degree of a
unipotent character $\chi$ is $q^{N-A_\chi}$, with $A_\chi$
the degree of the degree polynomial of the Alvis--Curtis dual of $\chi$ (see
\cite[Prop.~3.4.21]{GM20}). Lower bounds for $A_\chi$, with $\chi$ not the
Steinberg character, are as given in Table~\ref{tab:Achi}. (They do not
depend on the action of $F$ on the Weyl group.)   \par
The maximal power of $q$ dividing the degree of a non-unipotent character is,
by Lusztig's Jordan decomposition (see \cite[Thm~2.6.22]{GM20}), the order
$q^{N_1}$ of a Sylow $p$-subgroup of the centraliser of some non-trivial and
hence non-central semisimple element~$s$ in the dual group $G^*$ (which is of
adjoint type). Now note that $N$ is the number of reflections in
the Weyl group of $\bG$, while $N_1$ is the corresponding number in the Weyl
group of the centraliser $\bC_{\bG^*}(s)$. By inspection of possible reflection
subgroups of the various types of irreducible Weyl groups, $N-N_1$ is at least
as given in Table~\ref{tab:Achi}.   \par
In the last line of that table we bound the exponent of
$P$, that is, we bound the maximal $m$ such that $P$ contains an element of
order $p^m$, in any characteristic~$p$. In classical groups this is easily
derived from the existence of the classical matrix representation, and for
exceptional groups it is well-known. By inspection, this exponent of $P$ is
bounded above by both $p^{A_\chi}$ and $p^{N-N_1}$, whence our claim.
\end{proof}

\begin{cor}   \label{cor:def aut}
 Conjecture~{\rm\ref{conj:aqs12}} holds for groups $H$ whose non-abelian
 simple composition factor is of Lie type for the defining prime $p$.
\end{cor}

\begin{proof}
By Proposition~\ref{prop:spor} we need not concern ourselves with the groups
having exceptional Schur multipliers, nor with the Tits group. Let $\bG$ be
simple of simply connected type and $F:\bG\to\bG$ a Steinberg endomorphism
such that $L=[H,H]$ is a central quotient of $G=\bG^F$. Since $\bZ(G)$, and
hence $\bZ(L)$, has $p'$-order, the claim follows with
Proposition~\ref{prop:to qs} from Proposition~\ref{prop:defchar}.
\end{proof}

\subsection{Groups of Lie type in cross characteristic}
We now consider our conjectures for groups of Lie type at primes $\ell$
different from the defining characteristic, and prove the following theorem.
As customary, we write $\GL_n(-q):=\GU_n(q)$, and so forth.

\begin{thm}   \label{thm:ell>7}
 Let $S$ be quasi-simple of Lie type and $\ell>2$ a prime different from the
 defining characteristic of $S$. Assume that $\ell{\not|}(n,q-\eps)$ when
 $S/\bZ(S)=\PSL_n(\eps q)$ with $\eps\in\{\pm1\}$, and $\ell>5$ if $S$ is of
 exceptional type. Then all $\ell$-blocks of $S$ satisfy both conditions
 $(\ddagger)$ and $(\ddagger^\star)$. In particular
 Conjecture~{\rm\ref{conj:aqs12}} holds at such primes $\ell$ whenever
 $[H,H]=S$.
\end{thm}

A first reduction towards the proof of Theorem~\ref{thm:ell>7} is given as
follows:

\begin{prop}   \label{prop:abelian}
 The assertion of Theorem~{\rm\ref{thm:ell>7}} holds whenever $\ell$ does not
 divide the order of the Weyl group of~$S$. In particular, it holds whenever
 $S$ is of exceptional Lie type and $\ell>7$.
\end{prop}

\begin{proof}
By Proposition~\ref{prop:spor} we may assume that $S$ does not have exceptional
covering groups. Then since $\ell$ does not divide the order of the Weyl group
of $S$, Sylow $\ell$-subgroups of $S$ are abelian (see
\cite[Thm~25.14]{MT}), so our claim follows from Corollary~\ref{cor:ab}.
Since the Weyl groups of exceptional groups of Lie type have order not
divisible by primes $\ell>7$, they are all covered by the preceding
observation.
\end{proof}

For the classical types, we first discuss unipotent blocks.

\begin{prop}   \label{prop:unip class}
 Let $G=\GL_n(q)$, $\GU_n(q)$, $\Sp_{2n}(q)$, or $\SO_n^{(\pm)}(q)$ ($n\ge7$).
 Then Condition~$(\ddagger)$ holds for all unipotent $\ell$-blocks of $G$ for
 all primes $\ell>2$ not dividing~$q$.
\end{prop}

\begin{proof}
Let $G$ be one of the groups in the statement, and let $\bG$ be a connected
reductive group over $\overline{\FF}_q$ with a Frobenius endomorphism $F$ such
that $G=\bG^F$. Let $B$ be a unipotent $\ell$-block of $G$. Now $\ell$ is good
for $\bG$, so by the main result of \cite{CE94} there exists a unipotent
$d$-cuspidal pair $(\bL,\la)$ of~$\bG$, where~$d$ is the order of $q$
modulo~$\ell$, such that $\Irr(B)$ is described as follows: for any
$\ell$-element $t\in G^*=\bG^{*F}$, a character $\chi\in\cE(G,t)$ lies in
$\Irr(B)$ if and only if its Jordan correspondent
$\chi_t\in\cE(\bC_{G^*}(t),1)$, where $(\bG^*,F)$ is dual to $(\bG,F)$, lies
in the $d$-Harish-Chandra series of a
unipotent $d$-cuspidal pair $(\bL_t^*,\la_t^*)$ of $\bC_{\bG^*}(t)$, with
$[\bL,\bL]=[\bL_t,\bL_t]$ and $\la,\la_t$ agreeing on $[\bL,\bL]^F$.
(Here note that $\bC_{\bG^*}(t)$ is a Levi subgroup of~$\bG^*$ since $\ell$ is
good for $\bG$ and the group of components of $\bZ(\bG)$ has order prime to
$\ell$, and thus there exist dual Levi subgroups $\bC_{\bG^*}(t)^*$ and $\bL_t$
of~$\bG$.)
Further, any Sylow $\ell$-subgroup of $\bC_\bG^\circ([\bL,\bL])^F$ is a defect
group of $B$. We note that by Lusztig's Jordan decomposition the degrees are
related by
$\chi(1)=|G^*:\bC_{G^*}(t)|_{q'}\,\chi_t(1)$, and so
$$|G|_\ell/\chi(1)_\ell=|\bC_{G^*}(t)|_\ell/\chi_t(1)_\ell,$$
that is, $\chi$ and $\chi_t$ have the same defect.

First assume $G=\GL_n(q)$. Since $\bL$ is a $d$-split Levi subgroup of $\GL_n$
with a $d$-cuspidal unipotent character, it has the rational form
$\bL^F\cong (q^d-1)^b\GL_{n-bd}(q)$ for some $b\ge0$ (see
\cite[Exmp.~3.5.29]{GM20}), and thus $\bC_G([\bL,\bL])\cong\GL_{bd}(q)$. We may,
and will in fact, assume $b>0$ since otherwise $B$ is of defect zero. Now a
Sylow $\ell$-subgroup of $\GL_{bd}(q)$ lies in the normaliser of one of its
Sylow $d$-tori \cite[Cor.~25.17]{MT}, so in an extension of a homocyclic group
$C_{q^d-1}^b$
by the relative Weyl group $W_d\cong C_d\wr\fS_b$ \cite[Exmp.~3.5.29]{GM20}.
Since clearly the order of $\ell$-elements in $\fS_b$ is bounded above by~$b$,
the exponent of a defect group of $B$ is bounded above by $b\ell^r$, where
$\ell^r:=(q^d-1)_\ell$ is the exact power of $\ell$ dividing $q^d-1$.

Let $t\in G^*\cong G=\GL_n(q)$ be an $\ell$-element. Considering the
eigenspaces of~$t$ it is easily seen that
$$\bC_{G^*}(t)\cong\GL_{n_0}(q)\times\prod_{i} \GL_{n_i}(q^{d\ell^{e_i}})$$
with $n_0+d\sum_i n_i\ell^{e_i}=n$, where $i$ runs over minimal polynomials of 
non-trivial $\ell$-elements in $\overline{\FF}_q^\times$, $d\ell^{e_i}$ is the
degree of $i$ and $n_i$ the dimension of the associated eigenspace. 

For the Lusztig series $\cE(\bC_{G^*}(t),1)$ to contain elements from a
$d$-Harish-Chandra series $(\bL_t^*,\la_t^*)$ with $[\bL_t,\bL_t]=[\bL,\bL]$ we
certainly need $\bL_t^*\le \bC_{\bG^*}(t)$ (up to conjugation). Now the only
$d$-split Levi subgroups of $\GL_{n_i}(q^{d\ell^{e_i}})$ with a $d$-cuspidal
unipotent character are the maximally split tori \cite[Exmp.~3.5.29]{GM20},
so in fact $\bL_t\le \GL_{n_0}\bT$, for $\bT$ a suitable torus of
$\bC_{\bG^*}(t)$, and hence $n-bd\le n_0$.

Let $\chi\in\cE(G,t)\cap\Irr(B)$ and $\chi_t\in\cE(\bC_{G^*}(t),1)$ its
Jordan correspondent. Then $\chi_t=\chi_0\boxtimes\prod_i\chi_i$ with
$\chi_0\in\cE(\GL_{n_0}(q),1)$ and $\chi_i\in\cE(\GL_{n_i}(q^{d\ell^{e_i}}),1)$.
Now by \cite[Prop.~6.5]{KMS} the height of $\chi_i$, $i\ne0$, is bounded above
by the maximal height in the Weyl group $\fS_{n_i}$, so certainly by
$|\fS_{n_i}|_\ell$, while
$|\GL_{n_i}(q^{d\ell^{e_i}})|_\ell=\ell^{n_i(r+e_i)}\,|\fS_{n_i}|_\ell$.
Thus
$$|\GL_{n_i}(q^{d\ell^{e_i}})|_\ell/\chi_i(1)_\ell\ge \ell^{n_i(r+e_i)},$$
and similarly,
$$|\GL_{n_0}(q)|_\ell/\chi_0(1)_\ell\ge \ell^{rx}$$
with $x:=(n_0-n+bd)/d$. We thus have $\df(\chi)\ge rx+\sum_i n_i(r+e_i)$.
Write $y:=x+\sum_{e_i=0}n_i$. First assume $b>y\ge x\ge0$, and hence
there is at least one $i$ with $n_ie_i>0$. With the trivial estimate
$$\ell^{\sum_{e_i>0} n_ie_i}=\prod_{e_i>0} \ell^{n_ie_i}
  \ge\sum_{e_i>0} \ell^{n_ie_i} \ge\sum_{e_i>0} n_i\ell^{e_i}=b-y$$
we obtain
$$\ell^{rx+\sum_i n_i(r+e_i)} \ge\ell^{ry+r\sum_{e_i>0}n_i}(b-y)
  \ge\ell^{r(y+1)}(b-y)=\ell^r\cdot \ell^{ry}(b-y).$$
Since $(\ell^r)^y(b-y)\ge b$ for all $\ell^r\ge3$, $b>y$, we are done in this
case. Finally assume $b=y$. If $b=1$ then defect groups of $B$
are abelian, so we have $b\ge2$. In this case $\ell^{rb}\ge \ell^rb$ and again
our claim follows.
\par
The argument for $\GU_n(q)$ is entirely analogous, replacing $q$ by $-q$
throughout (so now $d$ is the order of $-q$ modulo $\ell$,
$\GL_{n_0}(-q)=\GU_{n_0}(q)$ and so on).   \par
If $G$ is one of the other classical groups we let $d$ be the order of $q$
and $e$ the order of $q^2$ modulo~$\ell$. Here the $d$-split Levi subgroups
possessing a $d$-cuspidal unipotent character have the rational form
$\bL^F=(q^e+(-1)^d)^a H(q)$, where $H$ is a classical group of the same
(or possibly twisted) type as $G$ of rank $n-ae$ (see
\cite[Exmp.~3.5.29(b)]{GM20}).

Assume first that $d$ is odd
and so $d=e$. Then centralisers of $\ell$-elements $t\in G^*$ have the form
$$\bC_{G^*}(t)\cong G_{n_0}(q)\times\prod_{i} \GL_{n_i}(q^{d\ell^{e_i}})$$
with $n_0+d\sum_i n_i\ell^{e_i}=n$ where $G_{n_0}$ has the same type as $G^*$
and rank $n_0$ and again $i$ runs over minimal polynomials of non-trivial
$\ell$-elements, of degree $d\ell^{e_i}$. According to \cite[Thm (ii)]{CE94}
a defect group of $B$ is contained in $\bC_\bG^\circ([\bL,\bL])^F$, which in
this case is a group of the same type as~$G$ and rank~$ad$. Setting again
$x:=(n_0-n+ad)/d$ we get $\df(\chi)\ge rx+\sum_i n_i(r+e_i)$ for any
$\chi\in\cE(G,t)\cap\Irr(B)$, where $\ell^r=(q^d-1)_\ell$, while the exponent
of a defect group is bounded above by $a\ell^r$, so we may conclude by exactly
the same estimates as before.

In the case when $d$ is even and so $e=d/2$, centralisers of $\ell$-elements
$t\in G^*$ have the form
$$\bC_{G^*}(t)\cong G_{n_0}(q)\times\prod_{i} \GU_{n_i}(q^{e\ell^{e_i}})$$
with $n_0+e\sum_i n_i\ell^{e_i}=n$ where $G_{n_0}$ has the same type as $G^*$
(or possibly type $\SO^{-\eps}$ if $G$ is of type $\SO^\eps$) and rank $n_0$,
and $i$ runs over minimal polynomials of non-trivial $\ell$-elements over
$\FF_{q^2}$, of degree $e\ell^{e_i}$. Again by \cite[Thm (ii)]{CE94}, a defect
group of $B$ is contained in $\bC_\bG^\circ([\bL,\bL])^F$, a group of rank~$ae$.
The argument is now entirely similar to the one in the previous case.
\end{proof}

We will need a slight improvement of the previous result for linear and unitary
groups:

\begin{prop}   \label{prop:unip GL}
 Let $G=\GL_n(\eps q)$ with $\eps\in\{\pm1\}$, $\ell>2$ a prime dividing
 $q-\eps$ and suppose $n$ is not a power of $\ell$. Then for all characters
 $\chi$ in the (unique) unipotent $\ell$-block $B$ of $G$
 $$(q-\eps)_\ell\,\chi(1)_\ell\,\exp(D) \quad\text{divides}\quad |G|_\ell\,,$$
 where $D$ is a defect group of $B$. In particular, under the stated
 assumptions, the (unique) unipotent $\ell$-block of $\SL_n(\eps q)$ satisfies
 $(\ddagger)$.
\end{prop}

\begin{proof}
We discuss the case $G=\GL_n(q)$, so $\eps=1$, the unitary case being entirely
similar. Since $\ell|(q-1)$ the only unipotent block of $G$ is the principal
block $B$, associated to the trivial unipotent $1$-cuspidal pair $(\bT,1)$,
where $\bT$ is a maximally split torus of $\bG$ (e.g.\ by \cite{CE94}). So any
defect group $D$ of $B$ is a Sylow $\ell$-subgroup of~$G$, of exponent bounded
above by $\ell^rn$, with $\ell^r=(q-1)_\ell$, since $D$ has a conjugate in the
torus normaliser $\bT^F.\fS_n$.

We now follow the argument in the proof of Proposition~\ref{prop:unip class}.
If $t\in G^*=\GL_n(q)$ is an $\ell$-element then
$$\bC_{G^*}(t)\cong \GL_{n_0}(q)
  \times\prod_{i}\GL_{n_i}(q^{\ell^{e_i}})$$
for suitable $n_i\ge1$, $e_i\ge0$ with $n_0+\sum_{i} n_i\ell^{e_i} =n$, where
again $i$ runs over non-constant irreducible polynomials over $\FF_q$ of
$\ell$-power degree $\ell^{e_i}$. Let $\chi\in\cE(G,t)$. Then as seen in
Proposition~\ref{prop:unip class} we have
$\df(\chi)\ge rn_0+\sum_{i} n_i(r+e_i)$.

If $\sum_i n_i\ge2$ then $r+\df(\chi)\ge r\log_\ell n$ by the calculations in
the proof of Proposition~\ref{prop:unip class}, yielding our claim in this
case. Next, when $\sum_{i} n_i=1$, that is, there is just one non-trivial factor
$\GL_1(q^{\ell^{e_i}})$, then our inequality still holds when $n_0>1$.
If $n_0=1$ then in fact the exponent of a Sylow $\ell$-subgroup of~$G$ is
bounded above by $\ell^r(n-1)=\ell^{r+e_1}$ and the asserted inequality holds
(as an equality). The case $n_0=0$ can only happen when
$n=\ell^{e_1}$ is an $\ell$-power, which was excluded. Finally, when
$\sum_i n_i=0$ then $n_0=n$, $t$ is central and we are fine when $n\ge3$.
For $n=2$ the defect groups are abelian as $\ell\ne2$, so again we are done.

For the final claim note that $G_0=\SL_n(q)$ is a normal subgroup of $\GL_n(q)$
of index $q-1$, and again the principal block $B_0$ is the only unipotent
$\ell$-block of $G_0$. The exponent of a Sylow $\ell$-subgroup $D_0$ of $G_0$
is of course bounded above by the one of~$D$, and the degrees of all characters
in $B_0$ divide some character degree in $B$ by Clifford theory, hence the
assertion follows from the proven inequality for $\GL_n(q)$.
\end{proof}

\begin{rem}   \label{rem:GL}
 Is is clear from the proof of Proposition~\ref{prop:unip GL} that the
 assumption on $n$ is only used to rule out the possibility that
 $\bC_{G^*}(t)=\GL_1(q^n)$, with $n=\ell^e>1$ an $\ell$-power. In that case,
 $t$ is a regular element in a Coxeter torus $T\cong\FF_{q^n}^\times$
 of~$\GL_n(q)$, and $\cE(G,t)=\{\pm R_T^G(\tht)\}$ for some $\tht\in\Irr(T)$ of
 $\ell$-power order.
 Note that the order of $t$ must be $(q^n-1)_\ell$, the precise power of $\ell$
 dividing $q^n-1$, since otherwise $t$ lies in a subtorus of order
 $q^{n/\ell}-1$, has centraliser containing $\GL_\ell(q^{n/\ell})$ and so
 cannot be regular.
 Thus, even when $n=\ell^e$ all characters in the principal block of $\GL_n(q)$
 satisfy the stronger inequality from Proposition~\ref{prop:unip GL} except
 only for $\chi=\pm R_T^G(\tht)$ with $\tht\in\Irr(T)$ of order $(q^n-1)_\ell$.
 This situation, of course, leads to Example~\ref{gunterexmp}.
\end{rem}

\begin{prop}   \label{prop:dagger class}
 Let $G$ be one of $\GL_n(q)$, $\GU_n(q)$, $\Sp_{2n}(q)$ with $n\ge2$, or
 $\Spin_n^{(\pm)}(q)$ with $n\ge7$, and let $\ell>2$. Then all $\ell$-blocks
 of~$G$ satisfy Condition~$(\ddagger)$.
\end{prop}

\begin{proof}
There exists a connected reductive algebraic group $\bG$ over
$\overline{\FF}_q$ with one simple component and a Frobenius map $F:\bG\to\bG$
such that $G=\bG^F$.
Let $B$ be an $\ell$-block of $G$. Then there is a semisimple $\ell'$-element
$s\in G^*$ such that $\Irr(B)\subseteq\cE_\ell(G,s)$ (see \cite[Thm~9.12]{CE}).
Let $\bL^*\le\bG^*$ be a minimal $F$-stable Levi subgroup with
$\bC_{\bG^*}^\circ(s)\le\bL^*$, and let $\bL\le\bG$ be dual to $\bL^*$. Then by
\cite[Thm~1.1]{BDR} there is a group~$N$ containing~$\bL^F$ with index dividing
$a(s):=|\bC_{\bG^*}(s)^F:\bC_{\bG^*}^\circ(s)^F|$, an $\ell$-block $b_N$ of $N$
with defect groups isomorphic to those of $B$, and a height preserving
bijection from
$\Irr(B)$ to $\Irr(b_N)$ (in fact, in most cases there is even a Morita
equivalence between $B$ and $b_N$, but we don't need this here).
Thus it suffices to prove $(\ddagger)$ for $b_N$. Let $b$ be a block of~$\bL^F$
covered by $b_N$. By \cite[Prop.~14.20]{MT}, the index $a(s)$
divides $o(s)$ and so is prime to~$\ell$, thus by Lemma~\ref{lem:updown} we
are done if we can show $b$ satisfies $(\ddagger)$.

First assume that $\bL$ is a proper Levi subgroup of $\bG$ (so $s$ is not
isolated in~$\bG^*$). If $G$ is one of $\GL_n(q)$, $\GU_n(q)$ or $\Sp_{2n}(q)$
then $\bL$ is a direct product of groups of type $\GL_m$, $\GU_m$, and
$\Sp_{2m}$, with $m<n$, possibly over extension fields of~$\FF_q$. By
induction over the rank, $(\ddagger)$ holds for all $\ell$-blocks of any such
direct factor, and hence also for~$b$. (Note that $(\ddagger)$ trivially holds
for the cyclic groups $\GL_1(q)$ and $\GU_1(q)$.) Now assume $G$ is a spin
group. The Levi subgroups of $\SO_{2n+1}(q)$ and $\SO_{2n}^\pm(q)$ are again
direct products of groups of type $\GL_m$, $\GU_m$, $\SO_{2m+1}$ or
$\SO_{2m}^\pm$ with $m<n$, possibly over extension fields, thus by induction
all of their $\ell$-blocks satisfy $(\ddagger)$. The group $\Spin_{2n+1}(q)$ is
a subgroup of index at most~2 of a central extension of degree at most~2 of
$\SO_{2n+1}(q)$, so the same relation holds between the Levi subgroups of these
two. Thus, by Lemma~\ref{lem:updown}, the blocks of proper Levi subgroups of
$\Spin_{2n+1}(q)$ also satisfy~$(\ddagger)$, and so $(\ddagger)$ holds for $b$.
The same argument applies to the groups $\Spin_{2n}^\pm(q)$. This completes
the proof for non-isolated elements $s$.

Now assume $s$ is isolated. If $s$ is central in $\bG^*$ (so $s=1$ or
$\bG=\GL_n$), the block $B$ is either unipotent or obtained from a unipotent
block of $G$ by tensoring with a linear character of $\ell'$-order, and so it
satisfies $(\ddagger)$ by Proposition~\ref{prop:unip class}. So now assume
$s\ne1$ and $\bG$ is not $\GL_n$. Here, up to extensions of 2-power order,
$\bC_{\bG^*}^\circ(s)$ is a product of two classical groups (see
\cite[Tab.~2]{Bo05}). By comparing the description of blocks for $G$ and for
the occurring products, as well as their defect groups, in \cite{CE99}, it
follows (as stated in \cite[Thm~1.6]{En08}) that there is a unipotent block
of $\bC_{\bG^*}^\circ(s)$ with characters of the same heights as in $B$ and
with isomorphic defect groups. So then again we may conclude by
Proposition~\ref{prop:unip class}.
\end{proof}

In most cases, the result of Proposition~\ref{prop:dagger class} descends to
the special linear and unitary groups:

\begin{prop}   \label{prop:dagger SL}
 Let $G=\SL_n(\eps q)$ with $\eps\in\{\pm1\}$, $n\ge2$, let $\ell>2$ be a prime
 dividing $q-\eps$ but not dividing $n$. Then all $\ell$-blocks of~$G$ satisfy
 Conditions~$(\ddagger)$ and $(\ddagger^\star)$.
\end{prop}

\begin{proof}
Again, we just give the details in case $\eps=1$, so $G=\SL_n(q)$; the unitary
group case is obtained by replacing $q$ by $-q$ throughout. Our proof now
essentially follows the same steps as for Proposition~\ref{prop:dagger class}.
Let $B$ be an $\ell$-block of $G$ and let $s\in G^*=\PGL_n(q)$ be a semisimple
$\ell'$-element such that $\Irr(B)\subseteq\cE_\ell(G,s)$. Let $\bL^*\le\bG^*$
be an $F$-stable Levi subgroup with $\bC_{\bG^*}^\circ(s)\le\bL^*$, minimal
with this property and let $b$ be the unipotent $\ell$-block of its dual~$\bL^F$
related to $B$ via the Bonnaf\'e--Dat--Rouquier construction \cite{BDR}. Again
by \cite[Prop.~14.20]{MT} the component group of~$\bC_{\bG^*}(s)$ has order
prime to $\ell$ and thus we may deal with $\bL^F$. Now
$$\bL^F=\big(\GL_{n_1}(q^{e_1})\times\cdots\times\GL_{n_r}(q^{e_r})\big)
        \cap\SL_n(q)$$
for suitable $n_i,e_i\ge1$ with $n=\sum_i n_ie_i$. Then $b$ is covered by a
block $\hat b=\hat b_1\boxtimes\cdots\boxtimes\hat b_r$ of
$\prod_i \GL_{n_i}(q^{e_i})$,
with unipotent blocks $\hat b_i$ of $\GL_{n_i}(q^{e_i})$. All of the $\hat b_i$
satisfy $(\ddagger)$ by Proposition~\ref{prop:unip class}. Since $\ell$ does
not divide $n$, there is some $n_i$ that is prime to~$\ell$,
so the corresponding block~$\hat b_i$ satisfies the stronger estimate in
Proposition~\ref{prop:unip GL}, and since $|\bZ(\GL_n(q))|=q-1$ we see that
$b$ must satisfy $(\ddagger)$. Here note that $(q^{e_i}-1)_\ell\ge (q-1)_\ell$.
Thus $B$ satisfies $(\ddagger)$, and hence also $(\ddagger^\star)$ as
$|\bZ(G)|$ is prime to $\ell$.
\end{proof}

We now turn to the exceptional groups of Lie type:

\begin{prop}   \label{prop:exc good}
 Condition~$(\ddagger)$ holds for all unipotent $\ell$-blocks of quasi-simple
 groups $S$ of exceptional Lie type for all good primes $\ell$.
\end{prop}

\begin{proof}
Let $\bG$ be simple of simply connected type with Frobenius endomorphism $F$
such that $S$ is a central quotient of $G=\bG^F$, which is possible since
by Proposition~\ref{prop:spor} we may assume that $S$ does not have exceptional
covering groups. By Proposition~\ref{prop:abelian} we are left to discuss
primes $\ell$ that are good for $\bG$ but do divide the order of the Weyl
group. The only cases are then $\ell=5$ for $G$ of types $E_6$, $\tw2E_6$ or
$E_7$, and $\ell=7$ for types~$E_7$ and~$E_8$. By Corollary~\ref{cor:ab} we may
then restrict to characters of $G$ in blocks with non-abelian defect groups.

For the primes $\ell$ under consideration, the only unipotent $\ell$-block with
non-abelian defect is the principal block $B_0$ in the case that $\ell|(q-1)$
when $G\ne\tw2E_6(q)$, or $\ell|(q+1)$ when $G\ne E_6(q)$ (see, e.g.,
\cite[Thm~1.2]{KM}).
Let $\eps\in\{\pm1\}$ with $\ell|(q-\eps)$. A Sylow $\ell$-subgroup $D$ of $G$
is then an extension of a homocyclic group of exponent $(q-\eps)_\ell$ by a
cyclic group of order~$\ell$, by the same reference, hence has exponent
$\exp(D)\le\ell(q-\eps)_\ell$. Let $\chi\in\Irr(B_0)$, so there is an
$\ell$-element $t\in\bG^{*F}$ with $\chi\in\cE(G,t)$. Now by using the
description of the principal block in \cite[Thm]{CE94} the Jordan correspondent
of $\chi$ has degree polynomial not divisible by $q-\eps$. Furthermore, any
$\ell$-element $t\in\bG^{*F}$ lies in a Sylow torus of $\bG^{*F}$ of
order $(q-\eps)^r$, or in a torus of order $(q^\ell-\eps)(q-\eps)^{r-\ell}$,
where $r$ is the rank of~$\bG$. Thus by the degree formula for Jordan
decomposition, $\chi$ has defect at least
$2\nu_\ell(q-\eps)\ge 1+\nu_\ell(q-\eps)$, respectively
$\nu_\ell(q^\ell-\eps)=1+\nu_\ell(q-\eps)$, hence at least $\nu_\ell(\exp(D))$,
as desired.
\end{proof}

The above result does not extend to bad primes:

\begin{exmp}   \label{exmp:G2}
Let $G=G_2(q)$ with $4|(q-1)$ and $\ell=2$. Then the cuspidal unipotent
character of $G$ denoted $G_2[1]$ lies in the principal $2$-block $B_0$ of~$G$
and has defect~3, while Sylow $2$-subgroups of $G$ have exponent at least
$(q-1)_2$, larger than $2^3$ for example when $q=17$, so $B_0$ does not, in
general, satisfy $(\ddagger)$ (but Conjecture~A still holds in this case, see
Proposition~\ref{prop:small exc} below).
\end{exmp}

\begin{prop}   \label{prop:exc 5}
 Condition~$(\ddagger)$ holds for all $\ell$-blocks of quasi-simple
 groups $S$ of exceptional Lie type for all primes $\ell\ge7$.
\end{prop}

\begin{proof}
By Proposition~\ref{prop:abelian} we may assume $\ell=7$ and $S$ is of type
$E_7$ or $E_8$. Let $\bG$ be a simple algebraic group with a Frobenius map~$F$
with respect to an $\FF_q$-structure such that $S$ is a central quotient of
$G=\bG^F$. As $|\bZ(\bG^F)|\le2$ it suffices by Lemma~\ref{lem:updown} to
consider $S=G$. Let $B$ be a 7-block of $G$ and $s\in\bG^{*F}$ a semisimple
$7'$-element with $\Irr(B)\subseteq\cE_7(G,s)$, and $(\bL,\la)$ be the
associated $d$-cuspidal pair where $d$ is the order of $q$ modulo~7
\cite{CE99}.  Again by
Proposition~\ref{prop:abelian} we may assume $B$ has non-abelian defect groups
and thus by \cite[Lemma~4.16]{CE99}, we have $d\in\{1,2\}$ and the relative
Weyl group of $(\bL,\la)$ has order divisible by~7. By the same reference, in
all of these cases, the defect groups of $B$ are extensions of a homocyclic
group of exponent $(q-\eps)_7$ by a group of order~7, hence have exponent at
most $7(q-\eps)_7$, where $\eps\in\{\pm1\}$ is such that $q\equiv\eps\pmod7$.

By Proposition~\ref{prop:exc good}, the claim holds if $B$ is unipotent.
Assume $B$ is
quasi-isolated non-unipotent. Then $\bC_{\bG^*}(s)$ is of type $A_7$ in
$G=E_7(q)$, respectively of type $A_8,A_7+A_1,D_8$ or $E_7+A_1$ in $G=E_8(q)$.
In either case the Sylow 7-subgroups of $\bC_{\bG^*}(s)$ have exponent bounded
by the same quantity as for the defect groups of $B$, and all of their
$7$-blocks satisfy $(\ddagger)$ by
Propositions~\ref{prop:dagger class} and~\ref{prop:dagger SL}, respectively
by induction. As Jordan decomposition gives a height preserving bijection
between $\Irr(B)$ and some block of $\bC_{\bG^*}(s)$, this shows our claim in
these cases.

So finally assume $B$ is not quasi-isolated. Then $\bC_{\bG^*}(s)$ has type
$A_6$ in type $E_7(q)$, respectively type $A_6,A_6+A_1,A_7,D_7$ or $E_7$
in $E_8(q)$. By \cite{BDR} there is a (quasi-isolated) block $b$ of
$\bC_{\bG^*}(s)$ with the same relevant numerical invariants as $B$.
Let first $G=E_7(q)$. Then $\bC_{\bG^*}(s)^F=A_6(\eta q)(q-\eta)$ for some
$\eta\in\{\pm1\}$. If $\eta\ne\eps$ then all 7-blocks of $\bC_{\bG^*}(s)^F$
satisfy $(\ddagger)$ by Proposition~\ref{prop:dagger SL} in conjunction with
Lemma~\ref{lem:updown}. If $\eta=\eps$ then since
$|\bZ(\bC_{\bG^*}(s)^F)|=q-\eps$, the defects of characters in $b$ are larger
by $\nu_7(q-\eps)$ than in the covered block of $\SL_7(q)$, and so
Proposition~\ref{prop:dagger class} again gives $(\ddagger)$ for $b$ and
hence for $B$.
Now let $G=E_8(q)$. If $\bC_{\bG^*}(s)$ has no factor $A_6(\eps q)$ then $b$
satisfies $(\ddagger)$ by Proposition~\ref{prop:dagger class}
and Lemma~\ref{lem:updown}, respectively by the case of $E_7$ just treated.
If $\bC_{\bG^*}(s)$ has a factor $A_6(\eps q)$ then again $q-\eps$ divides
$|\bZ(\bC_{\bG^*}(s)^F)|$ (by a {\sf Chevie}-computation \cite{MChev})
and so the claim follows as before.
\end{proof}

We can now show the main result of this section:

\begin{proof}[Proof of Theorem~\ref{thm:ell>7}]
Let $S$ be as in Theorem~\ref{thm:ell>7}. By Proposition~\ref{prop:spor} we may
assume $S$ has no exceptional covering group. Our assumptions then imply that
$\ell\nmid |\bZ(S)|$, and hence it suffices to prove Condition $(\ddagger)$
for $S$. Furthermore, by Proposition~\ref{prop:exc 5}, $S$ is not of
exceptional type. So there exists a simple classical algebraic group $\bG$ of
simply connected type over $\overline{\FF}_q$ with a Frobenius map $F$ such
that $S$ is a central quotient of $G:=\bG^F$. By our assumptions on $\ell$,
$|\bZ(G)|$ is prime to $\ell$, so by Lemma~\ref{lem:updown} it suffices to
consider $\ell$-blocks of $G$. If $\bG$ is of type $B_n$, $C_n$ or $D_n$ then
$G$ is as in Proposition~\ref{prop:dagger class} and the claim follows.

So we have that $G=\SL_n(\eps q)$. Again by
Proposition~\ref{prop:dagger class}, all $\ell$-blocks of
$\tilde G:=\GL_n(\eps q)$ satisfy $(\ddagger)$. If $\ell$ does not divide
$q-\eps=|\tilde G:G|$, the same holds for all $\ell$-blocks of~$G$ by
Lemma~\ref{lem:updown}. Finally, if $\ell|(q-\eps)$ then by assumption $\ell$
does not divide $n$, in which case our claim is contained in
Proposition~\ref{prop:dagger SL}.
\end{proof}

\subsection{Further cases}
We now consider some of the cases left open by Theorem~\ref{thm:ell>7}.
First, inspecting the proof of Proposition~\ref{prop:dagger SL} we can further
restrict the type of obstructing blocks for $\SL_n(\eps q)$:

\begin{cor}   \label{cor:SLn}
 Let $G=\SL_n(\eps q)$ with $\eps\in\{\pm1\}$, $n\ge3$, and let
 $2<\ell|(n,q-\eps)$. Then an $\ell$-block $B$ of $G$
 satisfies~$(\ddagger^\star)$ unless possibly when
 $\Irr(B)\subseteq\cE_\ell(G,s)$ for an $\ell'$-element
 $s\in G^*=\PGL_n(\eps q)$ with either
 $\bC_{G^*}^\circ(s)\cong\GL_r((\eps q)^{n/r})/C_{q-\eps}$ for some
 $\ell$-power $r|n$, or
 $$\bC_{G^*}^\circ(s)\cong(\GL_{n_1}(q^{e_1})\times\GL_{n_2}(q^{e_2}))/C_{q-\eps}$$
 for some $\ell$-powers $n_1,n_2>1$ with $n=n_1e_1+n_2e_2$.
\end{cor}

\begin{proof}
Let $\eps=1$ (again the case $\eps=-1$ is entirely analogous). In the notation
and by the arguments in the proof of Proposition~\ref{prop:dagger SL},
$B$ satisfies $(\ddagger)$ unless
$\bC_{\bG^*}^\circ(s)^F\cong\prod_{i=1}^m\GL_{n_i}(q^{e_i})/\bZ(\GL_n(q))$ with
$n=\sum_i n_ie_i$ and all $n_i>1$ powers of~$\ell$.  Now note that the exponent
of the defect group of $\hat b$ is the maximum of those of the~$\hat b_i$,
while the defects of characters of the $\hat b_i$ add up. Since
any defect in a unipotent block of $\GL_{n_i}(q^{e_i})$ is at least
$(q-1)_\ell\le|\bZ(\GL_{n_i}(q^{e_i}))|_\ell$, for each
$i=2,\ldots,m$ we get an additional factor $(q-1)_\ell$, and in fact by
Proposition~\ref{prop:unip GL} even an additional factor $(q-1)_\ell^2$ if
$n_i$ is not a power of $\ell$. Thus, $(\ddagger^*)$ will be satisfied for $B$
unless $m\le2$ and moreover all $n_i$ are powers of $\ell$, as claimed.
\end{proof}

Of course, Example~\ref{gunterexmp} shows that the first excluded case does
lead to true exceptions.

\begin{rem}   \label{rem:SL exc}
 Let us further examine the first type of $\ell$-blocks of $G=\SL_n(\eps q)$
 singled out in Corollary~\ref{cor:SLn}. We formulate the details for $\eps=1$.
 So let $s\in G^*=\PGL_n(q)$ be an $\ell'$-element with
 $C^*:=\bC_{G^*}^\circ(s)\cong\GL_r(q^{n/r})/C_{q-1}$ for some
 $\ell$-power $r|n$. By \cite{BDR} there is a height and defect group
 preserving bijection $\cE_\ell(G,s)\to\cE_\ell(N,1)$ for some extension
 $N$ of~$C$. Since $|\bC_{\bG^*}(s):\bC_{\bG^*}^\circ(s)|$ divides $o(s)$ and
 so is prime to~$\ell$,
 any character of $N$ lies over a character of $C$ of the same $\ell$-defect.
 Let $T$ be a Coxeter torus of $\GL_n(q)$ and $\bar T^*$ its image in $G^*$.
 Now Remark~\ref{rem:GL} shows that the offending characters in $\cE_\ell(C,1)$
 are labelled by $\ell$-elements $t$ in the Coxeter torus $\bar T^*\le C^*$ of
 $G^*$, of order $((q^{n/r})^r-1)_\ell/(q-1)_\ell=(q^n-1)_\ell/(q-1)_\ell$.
 Under Jordan decomposition these correspond to characters $\chi\in\cE(G,st)$
 with $st\in\bar T^*$ regular, where the order of $t$ equals the full
 $\ell$-part of $|\bar T^*|=(q^n-1)/(q-1)$. In particular, these characters
 $\chi$ are the constituents of the restriction to $\SL_n(q)$ of irreducible
 Deligne--Lusztig characters $R_T^G(\tht)$ of $\GL_n(q)$ with $\tht\in\Irr(T)$,
 where we can choose $\tht$ such that the $\ell$-part of its order equals the
 full $\ell$-part of $|T|=|T^*|=q^n-1$.
\end{rem}

We now aim to show that at least in some cases where $(\ddagger^\star)$ fails,
Wilde's conjecture still holds.

\begin{prop}   \label{prop:hc}
 Let $L = \SL(U)$ with $U = \FF_q^n$, $q=p^f$, $G = \langle L,g\rangle \rhd L$,
 and $V = \CC^d$ a $\CC G$-module. Suppose that 
 \begin{enumerate}[\rm(1)]
  \item there is an element $g_1\in\GaL(U)\cong\GL_n(q)\rtimes C_f$ such that
   the conjugations by $g$ and by $g_1$ induce the same automorphism of $L$;
   and
  \item $h \in L$ is a $p'$-element with $\bC_G(g) \leq \bC_G(h)$.
 \end{enumerate}
 Then at least one of the following statements holds.
 \begin{enumerate}[\rm(a)]
  \item There is an $h$-invariant non-zero proper subspace $U_1$ of $U$ such
   that, if $Q$ is the unipotent radical of $\Stab_L(U_1)$, then $g$ stabilises
   the $Q$-fixed point subspace $V^Q$ and moreover 
\begin{equation}\label{eq:hc}
   \Trace(g|V) = \Trace(g|V^Q).
\end{equation}
  \item The $\FF_q\langle h\rangle$-module $U$ decomposes as
   $m(W_1\oplus\ldots\oplus W_t)$, a direct sum of $t$ pairwise
   non-isomorphic irreducible submodules $W_1,\ldots,W_t$, each with
   multiplicity~$m$, transitively permuted by $g_1$.
 \end{enumerate}
\end{prop}

\begin{proof}
Fix a basis $(e_1,\ldots,e_n)$ of $U$. Then $\GaL(U) =\GL(U)\rtimes\langle\si \rangle$, where 
$$\si:\sum^n_{i=1}x_ie_i \mapsto \sum^n_{i=1}x_i^p e_i$$ 
for $x_i \in \FF_q$. 
First we show that if $A$ is any $\FF_q$-subspace of $U$, then so is $g_1(A)$.
It suffices to prove the claim in the case $g_1=\si$. Note that, if
$\dim_{\FF_q}A=k$, then there is an element $g_2\in\GL(U)$ such that
$A=g_2(A_0)$, where $A_0 := \langle e_1,\ldots,e_k\rangle_{\FF_q}$. Now we
have $\si(A_0) = A_0$, and 
$$\si(A) = \si g_2(A_0) = \si g_2 \si^{-1} \si(A_0) = g_3(A_0),$$
where $g_3:= \si g_2\si^{-1}\in\GL(U)$. Since $g_3(A_0)$ is an
$\FF_q$-subspace, the claim follows. 

Note that $\si(x u) = x^p \si(u)$ for any $u\in U$ and
$x\in \FF_q$. Since $g_1\in \GL(U)\rtimes\langle\si\rangle$, there is a
power $r$ of $p$ such that $g_1(x u)=x^r g_1(u)$ for all $u\in U$
and $x\in\FF_q$: if $g_1$ belongs to the coset $\GL(U)\sigma^i$ with
$0 \leq i \leq f-1$, then we can take $r=p^i$.

By (2), $g$ centralises the subgroup $J:= \langle h\rangle \leq L$. By (1),
the same holds for $g_1$. Suppose that $A$ is any submodule of the
$\FF_q J$-module $U$. By the above discussion, $g_1(A)$ is an $\FF_q$-subspace,
which is now $J$-invariant. Thus $g_1(A)$ carries the structure of an 
$\FF_qJ$-module. Fixing a basis $(a_1,\ldots,a_k)$ of $A$ (over $\FF_q$),
assume that $h$ acts on this basis via a matrix 
$X = (x_{ij})_{1 \leq i,j \leq k}$. Then 
$$h(g_1(a_j))= g_1(h(a_j)) = g_1\bigl(\sum^k_{i=1}x_{ij}a_i \bigr)
   = \sum^k_{i=1}g_1(x_{ij}a_i) = \sum^k_{i=1}x_{ij}^rg_1(a_i).$$
Thus $h$ acts on $g_1(A)$ in the basis $(g_1(a_1), \ldots,g_1(a_k))$ via the
matrix 
\begin{equation}\label{act10a}
  X^{(r)} := (x_{ij}^r).
\end{equation}   
In particular,
\begin{equation}\label{act10}
  g_1(A) \cong A 
\end{equation}
as $\FF_qJ$-modules if $r=1$, equivalently if $g_1 \in \GL(U)$; we will need
this observation in the sequel.

Now suppose that $A$ and $B$ are two $\FF_q J$-submodules of $U$ which are
isomorphic. Choosing bases $(a_1,\ldots,a_k)$ of $A$ and $(b_1,\ldots,b_k)$ of
$B$, respectively, in which $h$ acts via the same matrix $X$, by the preceding
paragraph, we see that $h$ acts via the same matrix $X^{(r)}$ in the bases 
$(g_1(a_1),\ldots,g_1(a_k))$ of $g_1(A)$ and 
$(g_1(b_1),\ldots,g_1(b_k))$ of $g_1(B)$. Thus 
\begin{equation}\label{act11}
  g_1(A) \cong g_1(B) \mbox{ as }\FF_qJ\mbox{-modules}.
\end{equation}

Recall that $J$ acts semisimply on $U$. Hence we can decompose
$U = \bigoplus^t_{i=1}m_iW_i$, where $W_1,\ldots,W_t$ are pairwise
non-isomorphic
irreducible $J$-modules and $m_i\in\ZZ_{\geq 1}$. By \eqref{act11}, $g_1$
permutes the subspaces $m_1W_1,\ldots, m_tW_t$. If this action is transitive,
then (b) holds. Otherwise we get a decomposition $U = U_1 \oplus U_2$ as a
direct sum of $\FF_qJ$-modules, where $U_1, U_2 \neq 0$ (each being the sum of 
some $m_iW_i$) have no common irreducible summands and are both $g_1$-stable.
Setting 
$$G_1:= \langle L,g_1\rangle \leq \GaL(U)$$
now we show that $\bC_{G_1}(h)$ preserves each $U_i$. Indeed, since
$[g_1,h]=[g,h]=1$, we have $\bC_{G_1}(h) = \langle \bC_L(h),g_1\rangle$. By
construction, $U_1$ and $U_2$ are $g_1$-stable. On the other hand, \eqref{act10}
shows that $\bC_L(h) \leq L$ preserves each $m_iW_i$. Hence the claim follows. 

As $L \lhd G_1$, $\bC_{G_1}(h)$ normalises $P:=\Stab_L(U_1)$ and $Q:=\bO_p(P)$.
But $g$ and $g_1$ act the same on $L$, so $g$ also normalises $P$ and $Q$.
Suppose now that $u\in Q$ is centralised by $g$. Then
$u\in\bC_G(g)\leq \bC_G(h)$, i.e., $u \in \bC_L(h) \leq \bC_{G_1}(h)$. It
follows that $u$ stabilises both $U_1$ and $U_2$. But, as any element in $Q$,
$u$ acts trivially on both $U_1$ and $U/U_1$, so $u=1$. Thus $g$ acts
fixed-point-freely on $Q \smallsetminus \{1\}$. Since $Q$ is elementary
abelian, by duality we conclude that $g$ acts fixed-point-freely on
$\Irr(Q) \smallsetminus \{1_Q\}$. Now we can decompose 
$$V = V^Q \oplus \sum_{1_Q \neq \la \in \Irr(Q)}V_\la$$
as a $\CC Q$-module, where $V_\la$ affords the $Q$-character $\dim(V_\la)\la$,
$g$ stabilises $V^Q$ and permutes the set $\{V_\la \mid\la\neq 1_Q\}$ with
no fixed point. It follows that $\Trace(g|V) = \Trace(g|V^Q)$.
\end{proof}

\begin{prop}   \label{prop:sl2}
 Under the assumption of Conjecture \ref{conj:aqs12}, suppose that $L=\SL_n(q)$
 and $\chi_L\in\cE(L,s)$, where $s\in G^*=\PGL_n(q)$ is semisimple with
 $\bC_{L^*}(s)$ a Coxeter torus. Assume in addition that
 $2<\ell|\gcd(n,q-1)$, and that there is some element $h_1\in\GaL_n(q)$
 such that the conjugations by $h$ and $h_1$ induce the same automorphism
 of~$L$. Then Condition \eqref{condB} holds at $\ell$ for $(H,h,\chi)$.
\end{prop}

\begin{proof}
As noted in the proof of Lemma \ref{2aqs}, $\chi_L\in\Irr(L)$. Let $\chi$ be
afforded by a $\CC G$-module~$V$. Since $\chi$ is faithful,
$\bC_H(L) = \bZ(H)=\bZ(L)$. Let $a$ denote the order of the coset $hL$ in 
$H/L$, and let $u$ denote the $\ell$-part of $h^a\in L$. As $u$ is a power
of~$h$, $\bC_H(h) \leq \bC_H(u)$.  We now apply Proposition~\ref{prop:hc} to
$(H,h,u)$ in place of $(G,g,h)$. 

Suppose conclusion (a) of Proposition~\ref{prop:hc} holds. Since
$\chi(h)\neq0$, it follows that $V^Q \neq 0$, i.e., 
$\tw{*}R^L_M(\chi_L) \neq 0$ (where $M$ is a Levi complement to $Q$ in
$\Stab_L(U_1)$, in the notation of Proposition \ref{prop:hc}).
By \cite[Prop.~3.3.20]{GM20}, this implies $s$ is conjugate to some element
of the dual Levi subgroup $\mathbf{M}^*$ of $\bG^*$, that is, $M^*$ contains a
Coxeter torus of~$G^*$, since $s$ is regular with connected
centraliser a Coxeter torus. Hence $M$ contains a Coxeter torus $T$, of order
$(q^n-1)/(q-1)$, of $\SL_n(q)$. On the other hand, the assumption
$2 < \ell|\gcd(n,q-1)$ implies that $n \geq 3$ and $q \geq 4$.
In particular, $|T|$ is divisible by a primitive prime divisor of 
$q^n-1$ (see e.g. \cite[Thm~28.3]{MT}), which however does not divide the order
of any proper Levi subgroup of~$L$, contrary to $T \leq M$.

Thus conclusion (b) of Proposition~\ref{prop:hc} holds. Write
$$q=p^f,~\ell^b = (q-1)_\ell,~\ell^c = n_\ell,
  ~|\bZ(L)|_\ell = \ell^d \mbox{ where }d:=\min(b,c).$$
Changing $h_1$ to a suitable generator of $\langle h_1 \rangle$, we may assume
that there is some $r=p^{f/e}=q^{1/e}$ such that
$$h_1(x u) = x^r h_1(u)$$
for all $u\in\FF_q^n$ and $x\in\FF_q$. In particular, $h_1^e\in\GL_n(q)$,
and \eqref{act10} shows that the integer $t$ in Proposition~\ref{prop:hc}(b)
divides $e$. 

Recall from Remark \ref{rem:SL exc} that $\chi_L$ lies under a semisimple
character $\chi_{s_1}$ of $\GL_n(q)$, where $s_1\in\GL_n(q)$ is a regular
semisimple
element whose $\ell$-part has order $(q^n-1)_\ell = \ell^{b+c}$. Let $\la$
denote an eigenvalue of $s_1$. Since $\chi_L$ extends to $\chi$, $\chi_L$ is
invariant under $h$ and~$h_1$. Now $\chi_{s_1}$ and $\chi_{s_1}^{h_1}$ share a
common irreducible constituent $\chi_L$. It follows that $s_1$ and a
$\GL_n(q)$-conjugate of $s_1^{(r)}$ (using the notation of \eqref{act10a})
differ by a scalar $x\in\FF_q^\times$, and hence
$$\la^r= x\la^{q^i}$$
for some $0 \leq i \leq n-1$. Now if $i=0$, then $\la^{r-1} = x$, whence
$\la^{(r-1)(q-1)}=1$, and so $\ell^{b+c} = |\la|_\ell$ divides
$(r-1)_\ell(q-1)_\ell=\ell^b(r-1)_\ell$, implying $\ell^c|(r-1)$, and thus
\begin{equation}\label{div10}
  \ell^d|(r-1).
\end{equation}
Suppose $i >0$. Then $\la^{q^i-r}=x^{-1}$, whence $\la^{(q^i-r)(q-1)}=1$,
and so $\ell^{b+c}=|\la|_\ell$ divides
$(q^i-r)_\ell(q-1)_\ell=\ell^b(q^r-i)_\ell$. It follows that $\ell^c$ divides
$q^i-r=(q^i-1)-(r-1)$. As $b,c \geq d$ and both $q-1$ and $q^i-1$ are divisible
by $\ell^b$, we see that \eqref{div10} holds in this case as well.

Now we consider the transitive action of $h_1$ on the set of isomorphism types
of $W_1,\ldots,W_t$. Recall from the proof of Proposition~\ref{prop:hc}
that if $X$ denotes the action of $u$ on $W_1$ (in some basis), then the
$\ell$-element~$u$ acts on $h_1(W_1)$ via $X^{(r)}$ (in a suitable basis).
Let $\gamma$ denote an eigenvalue of $X$. Since the
$\FF_q\langle u\rangle$-module $W_1$ is irreducible, $\gamma$ has degree
$k:=\dim_{\FF_q}(W_1)$ over $\FF_q$, and the spectrum of the matrix $X$ is
$\{\gamma^{q^j} \mid 0 \leq j \leq k-1\}$. Moreover
$$\det(u|_{W_1}) = \det(X) = \gamma^{(q^k-1)/(q-1)}.$$
As $\det(X^{(r)}) = \det(X)^r$ and $u \in L$, we obtain
$$1=\det(u) = \prod^{t-1}_{i=1}\det\bigl(X^{(r^i)}\bigr)^m
   =\det(X)^{m\frac{r^t-1}{r-1}}
   =\gamma^{m\frac{q^k-1}{q-1}\cdot\frac{r^t-1}{r-1}}.$$
Using \eqref{div10} and the fact that $u$, and hence $\gamma$, is an
$\ell$-element, we see that $o(\gamma)$ divides the $\ell$-part of
$m\frac{q^k-1}{q-1}\cdot\frac{r^t-1}{r-1}$ which is ~$(mtk)_\ell$. Dimension
comparison shows that $n=mtk$. 
So we have shown that $o(\gamma)$ is at most $n_\ell = \ell^c$. 

Now let $\delta:=\gamma^{\ell^{c-d}}$. Then $\delta^{\ell^d}=1$.
Using~\eqref{div10} again, we see that $\delta^{r-1}=1$, and thus
$\delta^r=\delta=\delta^q$; in particular $\delta\in \FF_q$.
Thus $u^{\ell^{c-d}}$ acts on $\FF_q^n$ as the scalar matrix
$\diag(\delta,\delta,\ldots,\delta)$, and so $o(u\bZ(H))$ divides $\ell^{c-d}$.
We have therefore shown that
$$o(h\bZ(H))_\ell \mbox{ divides }|H/L|_\ell \ell^{c-d}.$$
On the other hand, 
$|L|/\chi(1)$ is a multiple of $|T|_\ell=(q^n-1)_\ell/(q-1)_\ell=n_\ell$, hence
$\ell^{\df(\chi)}=(|H|/\chi(1))_\ell$ is a multiple of $|H/L|_\ell n_\ell$.
So the $\ell$-part of $|H:\bZ(H)|/\chi(1)$ is a multiple of
$|H/L|_\ell \ell^{c-d}$, and we are done.
\end{proof}

To possibly extend Proposition \ref{prop:sl2} to groups $G$ that are extensions
of $[G,G]\cong\SL_n(q)$ by \emph{any} automorphism, as well as of 
$[G,G]\cong\SU_n(q)$ (and thus finish off the blocks singled out in
Remark~\ref{rem:SL exc}, one needs to extend the Curtis-type
formula~\eqref{eq:hc} to disconnected reductive algebraic groups and relate
constituents of Lusztig restrictions over $G$ and $[G,G]$. The former can be
done; however the latter seems much more difficult.
\medskip

Next we record the following result, which completes the case of $\GL_n(\eps q)$.

\begin{prop}   \label{prop:glu}
 Let $G = \GL_n(\eps q)$ with $\eps\in\{\pm1\}$. Suppose $\chi(g)\neq 0$ for
 some $\chi\in\Irr(G)$ and $g\in G$. Then $o(g\bZ(G))$ divides
 $|G/\bZ(G)|/\chi(1)$. In particular, both Conjecture~A and
 Conjecture~\ref{conj:aqs} hold for $G = \GL_n(\eps q)$, the latter whenever
 it is non-solvable.
\end{prop}

\begin{proof}
(1) We proceed by induction on $n$, with the induction base $n=1$ being
obvious. For the induction step, by a direct check (or arguing as in the proof
of Theorem \ref{thm:reduction}) when $\GL_n(\eps q)$ is solvable, and using
Corollary \ref{cor:def aut} when it is not, it suffices to
prove the statement for the $p'$-part of the order of $g\bZ(G)$, where $p|q$.

First we consider the case $\chi = \pm R^G_L(\psi)$ is Lusztig induced from a
proper Levi subgroup $L = \bC_G(t)$, for some semisimple element $t \in G^*=G$.
Note that $L$ is a direct product of subgroups of the form $\GL_m((\eps q)^a)$
with $m < n$, to which the induction hypothesis applies. Also write $g=su$ with
semisimple part $s$ and unipotent part $u$. Since $R^G_L(\psi)(g) \neq 0$, the
formula \cite[Thm~3.3.12]{GM20} for $R^G_L$ implies that there exist some
$h \in G$ and some unipotent element $v$ such that $s \in L^h$,
$v\in\bC_{L^h}(s)$, and $\psi^h(sv)\neq 0$. Applying the induction hypothesis
to each direct factor of $L^h$, we see that there is an integer $M$ dividing
$|L/\bZ(L)|/\psi(1)$ such that $(sv)^M \in \bZ(L)$. As $s$ is semisimple, there
is a $p'$-integer $N$ dividing $|L/\bZ(L)|_{p'}/\psi(1)_{p'}$ such that
$s^N \in \bZ(L)$.

Note that $\bZ(L)\geq\bZ(G)$; denote $d:= |\bZ(L)/\bZ(G)|$, so that
$s^{Nd} \in \bZ(G)$. Also, $\chi(1)=|G:L|_{p'}\psi(1)$, so $Nd$ divides 
$$d \cdot |L/\bZ(L)|_{p'}/\psi(1)_{p'} = |L/\bZ(G)|_{p'}/\psi(1)_{p'}
  =|G/\bZ(G)|_{p'}/\chi(1)_{p'}.$$
We have therefore shown that $o(g\bZ(G))_{p'}$ divides
$|G/\bZ(G)|/\chi(1)$.
\smallskip

(2) It remains to consider the case $\chi$ is not Lusztig induced from a
proper Levi subgroup, hence (see \cite[p.~116]{FS}) a unipotent character
times a linear character. Without loss we may assume that $\chi$ is the
unipotent character $\chi^\la$ labelled by a partition $\la\vdash n$. 
Its degree is given by the quantised hook formula
\begin{equation}\label{eq:glu1}
  \chi^\la(1)_{p'}=\pm \frac{(\eps q-1)\cdots((\eps q)^n-1)}{\prod_h((\eps q)^{\ell(h)}-1)},
\end{equation}
where $h$ runs over all the hooks of $\la=(\la_1\ge\ldots\ge \la_r)$, of length
$\ell(h)$ (see for example \cite[Prop.~4.3.1]{GM20}). It suffices to prove the
statement for the $r$-part, say $r^a$, of the order of $g\bZ(G)$, for any prime
$r \nmid q$. 

Note that $r^a$ divides the $r$-part of the order of some eigenvalue $\xi$ of
$s$ (acting on $\overline\FF_q^n$), with multiplicity say $m$. If $\eps=1$
then $\bC_G(s)$ is a direct product of subgroups $\GL_{m_i}(q^{t_i})$, and for
at least one $i$, we have $m_{i}=m$ and $\xi \in \FF_{q^{t_{i}}}^\times$.
Setting $t:=t_{i}$, we have that $r^a$ divides $q^t-1$. If $\eps=-1$, then
$\bC_G(s)$ is a direct product of subgroups $\GL_{m_i}(q^{t_i})$ with $2|t_i$,
or $\GU_{k_j}(q^{l_j})$ with $2 \nmid l_j$. Moreover, either 

$\bullet$ for at least one $i$, we have $m_i=m$ and
$\xi\in \FF_{q^{t_i}}^\times$, in which case, setting $t:=t_{i}$, we have that
$r^a$ divides $q^{t_{i}}-1=(-q)^t-1$, or

$\bullet$ for at least one $j$, we have $k_{j}=m$ and $\xi^{q^{l_{j}}+1}=1$, in
which case, setting $t:=l_{j}$, we have that $r^a$ divides
$q^{l_{j}}+1 = -((-q)^t-1)$.

In both cases, the assumption $\chi(g)\neq0$ and the formula in
\cite[Thm~(2G)]{FS} imply that $\la$ has at least one hook $h_1$ of length
$\ell(h_1)=t$. Since $r^a|((\eps q)^t-1)$ in both cases and $|\bZ(G)|=q-\eps$
divides $((\eps q)^{\ell(h_2)}-1)$ for another hook $h_2$ of $\la$, it follows
from \eqref{eq:glu1} that $r^a$ divides 
$$\begin{aligned}|G/\bZ(G)|_{p'}/\chi(1)_{p'} &
   = \pm\frac{1}{q-\eps}\prod_{h}((\eps q)^{\ell(h)}-1)\\
  & = \pm\frac{(\eps q)^{\ell(h_2)}-1}{q-\eps}((\eps q)^{\ell(h_1)}-1)
   \prod_{h \neq h_1,h_2}((\eps q)^{\ell(h)}-1),\end{aligned}$$
and the statement follows for the prime $r$.
\end{proof}

As an immediate consequence of Proposition~\ref{prop:glu}, Conjecture~A and
Conjecture \ref{conj:aqs} hold for the quasi-simple groups
$G\cong\SL_n(\eps q)$ at $\ell=2$ when either $n$ is odd or $q$ is even.
\medskip

Finally we discuss some small rank cases.

\begin{prop}   \label{prop:suzree}
 Condition~$(\ddagger)$ holds for all covering groups of $\tw2B_2(q^2)$ and
 $^2G_2(q^2)$ at all primes.
\end{prop}

\begin{proof}
For $\tw2B_2(q^2)$ the only prime dividing the order of the Weyl group is~2,
which is the defining prime, so we are done by Propositions~\ref{prop:defchar}
and~\ref{prop:abelian}. For $^2G_2(q^2)$ the only relevant non-defining prime is
again $\ell=2$, but here Sylow 2-subgroups are elementary abelian, so
$(\ddagger)$ holds trivially. Finally, the exceptional covering groups have
already been discussed in Proposition~\ref{prop:spor}.
\end{proof}

\begin{prop}   \label{prop:small exc}
 Conjecture A holds for all quasi-simple groups $S$ such that
 $S/\bZ(S)$ is one of $G_2(q)$ or $\tw2F_4(q^2)'$.
\end{prop}

\begin{proof}
By Proposition~\ref{prop:abelian} only the divisors $\ell$ of the order of the
Weyl group need to be considered, and furthermore only for $\ell{\not|}q$ by
Corollary~\ref{cor:def aut}. Thus, for $G=G_2(q)$, we have $\ell\le3$.
By \cite[\S7]{Hi90} the only 3-block of $G$ with non-abelian defect groups
is the principal block $B_0$. Let $\eps\in\{\pm1\}$ with $q\equiv\eps\pmod3$.
The Sylow 3-subgroups of $G$ have exponent at most $3(q-\eps)_3$.
By inspection of \cite{Hi90}, all characters in $\Irr(B_0)$ have defect at
least $2\nu_3(q-\eps)$, except for the three cuspidal unipotent characters
$G_2[1],G_2[\theta],G_2[\theta^2]$ when $\eps=1$,
respectively the three unipotent characters
$\phi_{2,1},G_2[\theta],G_2[\theta^2]$ when $\eps=-1$. All of these are of
defect~2. From the known character tables (see e.g. \cite[Anh.~B]{Hi90}) it
can be seen that these vanish on all elements of order divisible by~9.

For $\ell=2$ first consider the principal block. The Sylow 2-subgroups of $G$
lie in the normaliser of a homocyclic maximal torus of order $(q-\eps)^2$,
where $q\equiv\eps\pmod4$, thus have exponent at most $2(q-\eps)_2$. Then by
inspection all characters in $B_0$ satisfy $(\ddagger)$ except for the two
unipotent characters $G_2[1],G_2[-1]$ if $\eps=1$, respectively
$\phi_{2,1},\phi_{2,2}$ if $\eps=-1$, which have defect~3. By the character
table of $G$ in \cite[Anh.~B]{Hi90}, they vanish on all elements of order
divisible by~4. The other two unipotent blocks are of defect zero.
Again by \cite{Hi90} the non-unipotent blocks with non-abelian defect groups
are labelled by semisimple $2'$-elements with centraliser $\SL_3(\pm q)$ or
$\GL_2(q)$, and here $(\ddagger)$ is always satisfied by our earlier results.

For $G=\tw2F_4(q^2)$ we may assume $q^2\ne2$ by Proposition~\ref{prop:spor},
and only the prime $\ell=3$ is relevant. If $q^2\not\equiv-1\pmod9$, the Sylow
3-subgroups of $G$ have exponent 3, and we are
done. For $q^2\equiv-1\pmod9$, a Sylow 3-subgroup of $G$ is contained in the
centraliser $\SU_3(q^2)$ of a 3-central element and thus has exponent
$(q^2+1)_3$, while $|G|_3=3(q^2+1)_3^2$. By the description in
\cite[Bem.~2]{Ma90} all characters in the principal 3-block of $G$ have defect
at least $(q^2+1)_3$, so~$(\ddagger)$ is satisfied, except for three
cuspidal unipotent characters. The latter vanish on all elements whose order
is divisible by~9, by the table in \cite{Ma90}. Since there are no non-trivial
isolated $3'$-elements in $G$ all other 3-blocks are Morita equivalent to
3-blocks of some proper Levi subgroup of type $A_1$ or $\tw2B_2$ and thus
$(\ddagger)$ holds by our earlier results.
\end{proof}

With the above results, we can finally prove the last two theorems stated in the
introduction:

\begin{proof}[Proof of Theorem C]
This follows from Propositions~\ref{prop:spor}, \ref{prop:Sn},
\ref{prop:2Sn odd}, \ref{prop:defchar}, and Theorem~\ref{thm:ell>7}.
\end{proof}

\begin{proof}[Proof of Theorem D]
This follows from Propositions~\ref{prop:spor}, \ref{prop:defchar},
Theorem~\ref{thm:An}, and Proposition \ref{prop:suzree}.
\end{proof}


\end{document}